\documentclass[11pt]{amsart}
\usepackage{amsmath}
\usepackage{amssymb}
\usepackage{amsfonts}
\usepackage{amsthm}
\usepackage{enumerate}
\usepackage[mathscr]{eucal}
\usepackage{verbatim}
\usepackage{amsthm}
\usepackage{amscd}
\usepackage[mathscr]{eucal}
\usepackage{appendix}
\usepackage{tikz}
\usepackage{stmaryrd}

\numberwithin{equation}{section}

\newtheorem{innercustomgeneric}{\customgenericname}
\providecommand{\customgenericname}{}

\newcommand{\newcustomtheorem}[2]{\newenvironment{#1}[1]
  {\renewcommand\customgenericname{#2}
   \renewcommand\theinnercustomgeneric{##1}\innercustomgeneric}{\endinnercustomgeneric}}

\newcustomtheorem{customthm}{Theorem}

\newcommand{\newcustomlemma}[2]{\newenvironment{#1}[1]
  {\renewcommand\customgenericname{#2}
   \renewcommand\theinnercustomgeneric{##1} \innercustomgeneric}{\endinnercustomgeneric}}

\newcustomlemma{customlemma}{Lemma}

\newcustomlemma{customproposition}{Proposition}

\newcustomlemma{customclaim}{Claim}

\theoremstyle{plain}
\newtheorem{theorem}{Theorem}[section]
\newtheorem{lemma}[theorem]{Lemma}
\newtheorem{corollary}[theorem]{Corollary}
\newtheorem{proposition}[theorem]{Proposition}

\theoremstyle{remark}
\newtheorem{remark}{Remark}

\theoremstyle{definition}
\newtheorem{defn}{Definition}

\newcommand{\q}{\quad}
\newcommand{\qq}{\qquad}

\newcommand{\bbr}{\mathbb{R}}

\newcommand{\bbrn}{\mathbb R^n}

\newcommand{\bbc}{\mathbb{C}}

\newcommand{\bbn}{\mathbb{N}}

\newcommand{\xxxi}{\vec{\boldsymbol{\xi}\,}}
\newcommand{\xxx}{\vec{\boldsymbol{x}}}

\newcommand{\ppp}{\vec{\boldsymbol{p}}}

\newcommand{\uuu}{\vec{\boldsymbol{u}}}

\def\000{\vec{\boldsymbol{0}}}

\def\ii{{\mathrm{i}}}

\newcommand{\supp}{\mathrm{supp}}

\newcommand{\wh}{\widehat}

\newcounter{question}

\newcommand{\bpf}{\begin{proof}}

\newcommand{\epf}{\end{proof}}

\allowdisplaybreaks

\makeatletter
\@namedef{subjclassname@2020}{\textup{2020} Mathematics Subject Classification}
\makeatother

\makeindex         

\begin{document}

\author{Bae Jun Park}
\address{B. Park, Department of Mathematics, Sungkyunkwan University, Suwon 16419, Republic of Korea}
\email{bpark43@skku.edu}

\author{Naohito Tomita}
\address{N. Tomita, Department of Mathematics, Graduate School of Science, Osaka University, Toyonaka, Osaka 560-0043, Japan}
\email{tomita@math.sci.osaka-u.ac.jp}

\thanks{B. Park is supported in part by NRF grant 2022R1F1A1063637 and by POSCO Science Fellowship of POSCO TJ Park Foundation.  N. Tomita was supported by JSPS KAKENHI Grant Number 20K03700.}

\title[Sharp maximal function estimates for pseudo-differential operators]{Sharp maximal function estimates for linear and multilinear pseudo-differential operators}

\begin{abstract} 
In this paper, we study pointwise estimates for linear and multilinear pseudo-differential operators with exotic symbols in terms of the Fefferman-Stein sharp maximal function and Hardy-Littlewood type maximal function. 
Especially in the multilinear case, we use a multi-sublinear variant of the classical Hardy-Littlewood maximal function introduced by Lerner, Ombrosi, P\'erez, Torres, and Trujillo-Gonz\'alez \cite{Le_Om_Pe_To_Tr2009}, which provides more elaborate and natural weighted estimates in the multilinear setting.
\end{abstract}

\maketitle

%\tableofcontents

\section{Introduction}

 For a locally integrable function $f$ on $\bbrn$, we define the (homogeneous) sharp maximal function $\mathscr{M}^{\sharp}f$ by
 $$\mathscr{M}^{\sharp}f(x):=\sup_{Q:x\in Q}\inf_{c_Q\in \mathbb{C}}\bigg(\frac{1}{|Q|}\int_Q \big| f(y)-c_Q\big|\; dy\bigg)$$
 where the supremum is taken over all cubes in $\bbrn$ with sides parallel to the axes containing the point $x$. Similarly, we also define its inhomogeneous version by 
$$\mathcal{M}^{\sharp}f(x):=\sup_{Q:x\in Q, \ell(Q)\ge 1}\bigg(\frac{1}{|Q|}\int_Q\big| f(y)\big|\; dy\bigg)+\sup_{Q:x\in Q, \ell(Q)<1}\inf_{c_Q\in \mathbb{C}}\bigg(\frac{1}{|Q|}\int_Q \big| f(y)-c_Q\big|\; dy\bigg).$$
Clearly, it holds that
\begin{equation}\label{twomaxrel}
\mathscr{M}^{\sharp}f(x)\le \mathcal{M}^{\sharp}f(x), \q x\in\bbrn.
\end{equation}
 These sharp maximal functions were originally motivated by the definitions of the spaces $BMO$ and $bmo$ 
 with
 $$\Vert f\Vert_{BMO(\bbrn)}:=\big\Vert \mathscr{M}^{\sharp}f\big\Vert_{L^{\infty}(\bbrn)}\q \text{and} \q\Vert f\Vert_{bmo(\bbrn)}:=\big\Vert \mathcal{M}^{\sharp}f\big\Vert_{L^{\infty}(\bbrn)},$$
 but they are also very convenient tools to obtain pointwise estimates for many operators appearing in harmonic analysis. For example, if $T$ is a Calder\'on-Zygmund operator, then we have the pointwise inequality that for $1<r<\infty$
 \begin{equation}\label{sharphardy}
 \mathscr{M}^{\sharp}\big(Tf\big)(x)\lesssim  \mathrm{M}_rf(x)
 \end{equation}
 where $\mathrm{M}_r$ is the $L^r$ version of the Hardy-Littlewood maximal operator, defined by the formula
 \begin{equation}\label{hardymax}
 \mathrm{M}_rf(x):=\sup_{Q:x\in Q}\bigg(\frac{1}{|Q|}\int_Q |f(y)|^r\; dy \bigg)^{\frac{1}{r}}
 \end{equation}
 for  locally $r$-th power integrable functions $f$ on $\bbrn$,
 and we simply write $\mathrm{M}f(x)$ if $r=1$.
 Here and in the sequel, the symbol $A\lesssim B$ indicates that $A\le CB$ for some constant $C>0$ independent of the variable quantities $A$ and $B$, and $A\sim B$ if $A\lesssim B$ and $B\lesssim A$ hold simultaneously.
Moreover, if $1\le p_0\le p<\infty$ and $f\in L^{p_0}(\bbrn)$,  then
\begin{equation}\label{flpsharpmax}
\Vert f\Vert_{L^p(\bbrn)}\lesssim \big\Vert \mathscr{M}^{\sharp}f\big\Vert_{L^p(\bbrn)},
\end{equation} which was introduced by Fefferman and Stein \cite{Fe_St1972}.
  Due to \eqref{sharphardy} and \eqref{flpsharpmax}, we can relate $Tf$ to $\mathrm{M}_rf$ for $1<r<p$ and then the $L^p$ boundedness for $T$ follows from the maximal inequality for $\mathrm{M}_r$. To be specific,
  \begin{equation}\label{basicidea}
  \Vert Tf\Vert_{L^p(\bbrn)}\lesssim \big\Vert \mathscr{M}^{\sharp}(Tf)\big\Vert_{L^p(\bbrn)}\lesssim \big\Vert \mathrm{M}_rf \big\Vert_{L^p(\bbrn)}\lesssim\Vert f\Vert_{L^p(\bbrn)}.
  \end{equation}
 This idea appeared actually in a number of papers  in harmonic analysis; see  \cite{Ch_To1985, Fe_St1972, Jo2006, Ku_Wh1979, Mi1982, Ta2012}.
  Chanillo and Torchinsky \cite{Ch_To1985} employed such a technique to establish weighted norm inequalities for pseudo-differential oeprators.
  Given $m\in\bbr$ and $0\le \delta\le \rho\le 1$, the H\"ormander symbol class $S_{\rho,\delta}^m(\bbrn)$ consists of infinitely many differentiable functions $\sigma(x,\xi)$ on $\bbrn\times \bbrn$ such that for any multi-indices $\alpha, \beta\in (\bbn_0)^n:=\{0,1,2,\dots\}^n$ there exists a constant $C_{\alpha,\beta}$ such that
 $$\big| \partial_{x}^{\alpha}\partial_{\xi}^{\beta}\sigma(x,\xi)\big|\le C_{\alpha,\beta}\big( 1+|\xi|\big)^{m+\delta|\alpha|-\rho|\beta|}.$$
Let $\mathscr{S}(\bbrn)$ denote the Schwartz space on $\bbrn$.
 For a symbol $\sigma\in S_{\rho,\delta}^{m}(\bbrn)$, the corresponding (linear) pseudo-differential operator $T_{\sigma}$ is defined as
 $$T_{\sigma}f(x):=\int_{\bbrn}\sigma(x,\xi)\wh{f}(\xi) e^{2\pi i\langle x,\xi\rangle}\; d\xi$$
  for $f\in\mathscr{S}(\bbrn)$, where $\wh{f}(\xi):=\int_{\bbrn} f(x)e^{-2\pi i\langle x,\xi\rangle}\;dx$ is the Fourier transform of $f$.
 Denote by $\mathrm{Op}S_{\rho,\delta}^{m}(\bbrn)$ the class of pseudo-differential operators with symbols in $S_{\rho,\delta}^{m}(\bbrn)$.
 For $0\le \delta<  \rho\le 1$, the $L^2$-boundedness of $T_{\sigma}\in \mathrm{Op}S_{\rho,\delta}^{m}(\bbrn)$ was studied by H\"ormander \cite[Theorem 3.5]{Ho1967}
 and additionally the case $0\le \delta=\rho<1$ was treated by Calder\'on and Vaillancourt \cite{Ca_Va1972}. Later, their result was extended to $L^p$ space for $1<p<\infty$ by Fefferman \cite{Fe1973}.
   \begin{customthm}{A}\cite{Ca_Va1972, Fe1973}\label{thma}
Let $0<\rho<1$ and $1<p<\infty$.
Suppose that
\begin{equation*}%\label{degreecond}
m\le -n(1-\rho)\Big|\frac{1}{2}-\frac{1}{p} \Big|.
\end{equation*}
Then $\sigma\in S_{\rho,\rho}^{m}(\bbrn)$ satisfies
$$\big\Vert T_{\sigma}f\big\Vert_{L^p(\bbrn)}\lesssim \Vert f\Vert_{L^p(\bbrn)}$$
for $f\in\mathscr{S}(\bbrn)$.
 \end{customthm}
We also refer to \cite{Pa_So1988, Park2018} for the extension to local Hardy spaces $h^p$ for $0<p<\infty$ and Triebel-Lizorkin spaces $F_p^{s,q}$.

\hfill

 Recall Muckenhoupt's $A_p$ class of weights $w$, saying $w\in A_p$, for $1\le p<\infty$ if $w$ is a nonnegative locally integrable function satisfying
\begin{align*}
 \mathrm{M}w(x)\lesssim w(x) \q \text{ for almost all }~x\in\bbrn \qq\qq\qq& \text{ when } p=1,\\
 \sup_{Q:\text{cubes}}\Big( \frac{1}{|Q|}\int_Q w(x)\; dx\Big)\Big(\frac{1}{|Q|}\int_Q \big( w(x)\big)^{-\frac{1}{p-1}}\; dx \Big)^{p-1}<\infty \qq& \text{ when } 1<p<\infty.
 \end{align*}
For a weight $w$,
the weighted Lebesgue space $L^p(w)$, $0<p<\infty$, consists of all measureble functions $f$ on $\bbrn$ satisfying
$$\|f\|_{L^p(w)}:=\left(\int_{\bbrn}|f(x)|^p w(x)\, dx \right)^{\frac{1}{p}}<\infty . $$
It is known in \cite{Mu1972} that
 for $1<p<\infty$ 
  \begin{equation}\label{chaweight}
  w\in A_p \q\text{if and only if}\q \Vert \mathrm{M}f\Vert_{L^p(w)}\lesssim \Vert f\Vert_{L^p(w)}.
  \end{equation}
 For $p=\infty$, we define $A_{\infty}:=\bigcup_{p>1}A_p$ and then it turns out that $A_p\subset A_q$ if $1\le p\le q\le \infty$.

Now we state the result of Chanillo and Torchinsky \cite{Ch_To1985}.
 \begin{customthm}{B}\cite{Ch_To1985}\label{thmb}
Let $0\le \delta<\rho<1$ and $2\le p<\infty$. Suppose  $w\in A_{p/2}$ and
 $$\sigma\in S_{\rho,\delta}^{-\frac{n}{2}(1-\rho)}(\bbrn).$$
Then we have
 $$\Vert T_{\sigma}f\Vert_{L^p(w)}\lesssim \Vert f\Vert_{L^p(w)}$$
for $f\in \mathscr{S}(\bbrn)$.
 \end{customthm}
As mentioned above, the proof of Theorem \ref{thmb} is based on the pointwise maximal estimate such as \eqref{sharphardy}. 
Indeed, one of the key estimates in \cite{Ch_To1985} is the following pointwise inequality:
\begin{equation}\label{ChToptest}
\mathscr{M}^{\sharp}\big(T_{\sigma}f\big)(x)\lesssim \mathrm{M}_2f(x)
\end{equation}
for $f\in \mathscr{S}(\bbrn)$.
The inequality \eqref{ChToptest} has been extended to the case $0<\delta=\rho<1$ and $1<r\le 2$, as below, by Miyachi and Yabuta \cite{Mi_Ya1987}, but with a certain restriction on the range of $\rho$.
 \begin{customthm}{C}\cite{Mi_Ya1987}\label{thmc}
Let $1<r\le 2$, $0<\rho\le \frac{r}{2}$, and $\rho<1$.
Suppose that $m\le -\frac{n}{r}(1-\rho)$ and $\sigma\in S_{\rho,\rho}^{m}(\bbrn)$.
Then it holds that
$$\mathscr{M}^{\sharp}\big(T_{\sigma}f\big)(x)\lesssim \mathrm{M}_rf(x)$$
for $f\in \mathscr{S}(\bbrn)$.
 \end{customthm}

We also refer to \'Alvarez and Hounie \cite{Al_Ho1990},
Michalowski, Rule, and Staubach \cite{Mi_Ru_St2012}, and Wang \cite{Wa2022} for related results to Theorem \ref{thmc}.

The first main result of this article is a generalization of Theorem \ref{thmc} to $0<\rho<1$.
In order to state the result, we introduce the $L^r$ version of  the inhomogeneous sharp maximal function
 \begin{align*}
 \mathcal{M}_r^{\sharp}f(x)&:=\sup_{Q:x\in Q, \ell(Q)\ge 1}\bigg(\frac{1}{|Q|}\int_Q\big| f(y)\big|^r\; dy\bigg)^{\frac{1}{r}}\\
 &\q\q+\sup_{Q:x\in Q, \ell(Q)<1}\inf_{c_Q\in \mathbb{C}}\bigg(\frac{1}{|Q|}\int_Q \big| f(y)-c_Q\big|^r\; dy\bigg)^{\frac{1}{r}}
 \end{align*}
for $0<r<\infty$, where the supremum ranges over all cubes in $\bbrn$ containing $x$ whose side-length is $\ge 1$ in the first one, and is $< 1$ in the second one.
\begin{theorem}\label{firstmainthm}
Let $1<r\le 2$ and $0<\rho<1$. Suppose that $m\le -\frac{n}{r}(1-\rho)$ and $\sigma\in S_{\rho,\rho}^{m}(\bbrn)$. Then it holds that
$$\mathcal{M}_r^{\sharp}\big( T_{\sigma}f\big)(x)\lesssim \mathrm{M}_rf(x)$$
for $f\in \mathscr{S}(\bbrn)$.
\end{theorem}
\begin{remark}
We consider the inhomogeneous sharp maximal operator $\mathcal{M}_r^{\sharp}$ in Theorem \ref{firstmainthm} and the same conclusion also holds with $\mathcal{M}_r^{\sharp}$ replaced with homogeneous one $\mathscr{M}^{\sharp}$ in view of \eqref{twomaxrel}  and H\"older's inequality. 
\end{remark}
Then the following weighted norm inequality for $T_{\sigma}$ holds
as a corollary of Theorem \ref{firstmainthm}.
\begin{corollary}\label{p=rweight}
Let  $0\le \delta\le \rho<1$, $0<\rho<1$, $1<r\le 2$, and $r\le p<\infty$. Suppose that $m\le -\frac{n}{r}(1-\rho)$ and $\sigma\in S_{\rho,\delta}^{m}(\bbrn)$.
If $w\in A_{\frac{p}{r}}$, then
$$\Vert T_{\sigma}f\Vert_{L^p(w)}\lesssim \Vert f\Vert_{L^p(w)}$$
for $f\in \mathscr{S}(\bbrn)$.
\end{corollary}
The proof of the corollary is quite immediate except the case when $p=r$. Indeed, when $p>r$ and $w \in A_{p/r}$, similar to \eqref{basicidea},  we obtain
\begin{equation}\label{tsifpwflpw}
\Vert T_{\sigma}f\Vert_{L^p(w)}\lesssim \big\Vert \mathcal{M}_r^{\sharp}(T_{\sigma}f)\big\Vert_{L^p(w)}\lesssim \big\Vert \mathrm{M}_rf\big\Vert_{L^p(w)}\lesssim \Vert f\Vert_{L^p(w)}
\end{equation}
by applying \eqref{flpsharpmax}, \eqref{twomaxrel}, H\"older's inequality, and \eqref{chaweight}.
For $p=r$, we assume $w\in A_1$.
According to \cite[Corollary 7.6 (2)]{Duan}, there exists $\epsilon>0$ such that
$$w^{1+\epsilon}\in A_1.$$
For any $r<q_0<\infty$,
the embedding theory for $A_p$ weights implies $w^{1+\epsilon}\in A_{q_0/r}$ and thus
the inequality \eqref{tsifpwflpw} yields
$$\Vert T_{\sigma}f\Vert_{L^{q_0}(w^{1+\epsilon})}\lesssim \Vert f\Vert_{L^{q_0}(w^{1+\epsilon})}.$$
On the other hand,
it follows from Theorem \ref{thma} that 
$$\Vert T_{\sigma}f\Vert_{L^{q_1}(\bbrn)}\lesssim \Vert f\Vert_{L^{q_1}(\bbrn)}$$
 for all $1<q_1<r$
as $m\le -\frac{n}{r}(1-\rho)\le -n(1-\rho)(\frac{1}{q_1}-\frac{1}{2})$ due to $\frac{1}{r}\ge \frac{1}{2}>\frac{1}{q_1}-\frac{1}{2}$.
By choosing $q_0$ and $q_1$ appropriately and applying interpolation theory with \cite[Theorem 5.5.3]{Be_Lo}, the desired estimate follows.

\hfill

Now we turn our attention to multilinear pseudo-differential operators.
Let $l$ be a positive integer greater than $1$, which will serve as the degree of multilinearity of operators.
The multilinear H\"ormander symbol class $\mathbb{M}_lS_{\rho,\delta}^{m}(\bbrn)$ consists of all smooth functions $\sigma$ on $(\bbrn)^{l+1}$ having the property that for all multi-indices $\alpha, \beta_1,\dots,\beta_l\in (\bbn_0)^n$ there exists a constant $C_{\alpha,\beta_1,\dots,\beta_l}>0$ such that
$$\big| \partial_{x}^{\alpha}\partial_{\xi_1}^{\beta_1}\cdots\partial_{\beta_l}^{\beta_l}\sigma(x,\xi_1,\dots,\beta_l)\big|\le C_{\alpha,\beta_1,\dots,\beta_l}\big( 1+|\xi_1|+\cdots+|\beta_l|\big)^{m+\delta |\alpha|-\rho(|\beta_1|+\dots+|\beta_l|)},$$
 and let $T_{\sigma}$ (which is denoted as in the linear setting without risk of confusion as linear operators will be treated only in Section \ref{linearproof} and will not appear in the sequel) now denote the multilinear pseudo-differential operator associated with $\sigma\in \mathbb{M}_lS_{\rho,\delta}^{m}(\bbrn)$, defined by
 $$T_{\sigma}\big(f_1,\dots,f_l\big)(x):=\int_{(\bbrn)^l}\sigma(x,\xi_1,\dots,\xi_l)\prod_{j=1}^{l}\wh{f_j}(\xi_j) \, e^{2\pi i\langle x,\xi_1+\dots+\xi_l\rangle}\; d\xi_1  \cdots d\xi_l$$
for $f_1,\dots,f_l\in \mathscr{S}(\bbrn)$. We denote by $\mathrm{Op}\mathbb{M}_lS_{\rho,\delta}^{m}(\bbrn)$ the family of such multilinear operators.

In contrast to the $L^2$ estimate for the linear operator in Theorem \ref{thma}, which is especially known as Calder\'on-Vaillancourt theorem \cite{Ca_Va1972}, B\'enyi and Torres \cite{Be_To2004} pointed out that the degree condition $m=0$ does not assure any boundedness of the bilinear operators corresponding to $\sigma \in \mathbb{M}_2S_{0,0}^{m}(\bbrn)$. This difference between linear and bilinear cases has attracted much interest in the study of bilinear pseudo-differential operators.
The symbolic calculus of the bilinear operator was initially obtained in \cite{Be_Ma_Na_To2010} and after that, several boundedness properties of the bilinear operators have been established in a few papers.
For $0 \le \rho <1$ and $0<p_1,p_2 \le \infty$, we define
$$m_{\rho}(p_1,p_2):=-n(1-\rho)\bigg[\max\Big(\frac{1}{p_1},\frac{1}{2}\Big)+\max\Big(\frac{1}{p_2},\frac{1}{2}\Big)-\min\Big(\frac{1}{p},\frac{1}{2} \Big)\bigg]$$
where $1/p=1/p_1+1/p_2$.
In the subcritical case $m<m_{\rho}(p_1,p_2)$,
Michalowski, Rule, and Staubach \cite{Mi_Ru_St2014} proved that
if $0 \le \rho<1$ and $\sigma\in \mathbb{M}_2S_{\rho,\rho}^{m}(\bbrn)$, then
\begin{equation}\label{noconendcon}
\big\Vert T_{\sigma}(f_1,f_2)\big\Vert_{L^p(\bbrn)}\lesssim \Vert f_1\Vert_{L^{p_1}(\bbrn)}\Vert f_2\Vert_{L^{p_2}(\bbrn)}
\end{equation}
for $1\le p \le 2 \le p_1,p_2 \le \infty$,
and this result was extended to the range $1 \le p, p_1, p_2 \le \infty$ by
B\'enyi, Bernicot, Maldonado, Naibo, and Torres \cite{Be_Be_Ma_Na_To}.
More recently,
the boundedness result has been improved and further extended by Miyachi and Tomita \cite{Mi_To2013,  Mi_To2019, Mi_To2020} who established the bilinear estimate \eqref{noconendcon} in the critical case 
$$m=m_{\rho}(p_1,p_2)$$
for all $0<p,p_1,p_2\le \infty$ by properly replacing the Lebesgue space $L^p$ with (real) Hardy space $H^p$ when $0<p \le 1$ and with $BMO$ when $p=\infty$. For the  definitions and properties of $H^p$ and $BMO$, we refer to \cite{Fe_St1972} and \cite[Chapters 3 and 4]{St1993}.

  \begin{customthm}{D}\cite{Mi_To2013, Mi_To2019, Mi_To2020}\label{thmd}
 Let $0<p_1,p_2,p\le \infty$ with $1/p_1+1/p_2=1/p$ .
 Suppose that $0\le \rho<1$, $m\in\bbr$ and $\sigma\in \mathbb{M}_2S_{\rho,\rho}^{m}$.
If 
\begin{equation}\label{conditiononmle}
m\le m_{\rho}(p_1,p_2).
\end{equation}
then we have
$$\big\Vert T_{\sigma}(f_1,f_2)\big\Vert_{X^p(\bbrn)}\lesssim \Vert {f_1}\Vert_{H^{p_1}(\bbrn)}\Vert {f_2}\Vert_{H^{p_2}(\bbrn)}$$
for $f_i \in \mathscr{S}(\bbrn)\bigcap H^{p_i}(\bbrn)$, $i=1,2$,
where $H^{p}$ is the (real) Hardy spaces for $0<p<\infty$ and
we take the convention $H^{\infty}=L^{\infty}$ and
\begin{equation}\label{xpspacedef}
 X^p:=\begin{cases}
 L^p & 0<p<\infty\\
 BMO & p=\infty
 \end{cases}.
 \end{equation}
 \end{customthm}
We remark that the condition \eqref{conditiononmle} is a necessary condition for the boundedness. See \cite{Mi_To2013} for more details.

For general multilinear estimates, we have the following result which deals only with the case $\rho=0$.
  \begin{customthm}{E}\cite{Ka_Mi_To2022}\label{thme}
 Let $0<p_1,\dots,p_l\le \infty$ with $1/p_1+\dots+1/p_l=1/p$. Suppose that $m\in\bbr$ and $\sigma\in \mathbb{M}_lS_{0,0}^m(\bbrn)$.
 Then 
  $$m\le m_0(\ppp):=-n\bigg[ \sum_{j=1}^{l}\max\Big(\frac{1}{p_j},\frac{1}{2} \Big) -\min\Big(\frac{1}{p},\frac{1}{2} \Big)\bigg], \q \ppp:=(p_1,\dots,p_l)$$
 if and only if
 $$\big\Vert T_{\sigma}\big(f_1,\dots,f_l\big)\big\Vert_{Y^p(\bbrn)}\lesssim \prod_{j=1}^{l}\Vert f_j\Vert_{Y^{p_j}(\bbrn)}$$
 for $f_1,\dots,f_l\in \mathscr{S}(\bbrn)$,
 where $h^p$ denotes the local Hardy space, introduced by Goldberg \cite{Go1979}, and
 $$Y^p:=\begin{cases}
 h^p & 0<p<\infty\\
 bmo & p=\infty
 \end{cases}.$$
  \end{customthm}
Here, we point out that
\begin{equation*}
H^p\hookrightarrow h^p \q \text{and}\q bmo\hookrightarrow BMO.
\end{equation*}
For the case $0<\rho<1$, to the best of our knowledge, there is no multilinear extension of Theorem \ref{thmd} so far.  However, as a consequence of Theorem \ref{thme}, we can easily obtain the following boundedness result.
\begin{proposition}\label{multilineargerho}
Let $0<p_1,\dots,p_l\le \infty$ with $1/p_1+\dots+1/p_l=1/p$.
Suppose that $0<\rho<1$, $m\in\bbr$, and $\sigma\in \mathbb{M}_lS_{\rho,\rho}^m(\bbrn)$.
If 
$$m< m_{\rho}(\ppp):=(1-\rho)m_{0}(\ppp), \q \ppp:=(p_1,\dots,p_l),$$
then we have
$$\big\Vert T_{\sigma}\big(f_1,\dots,f_l\big)\big\Vert_{X^p(\bbrn)}\lesssim \prod_{j=1}^{l}\Vert f_j\Vert_{H^{p_j}(\bbrn)}$$
 for all $f_j\in \mathscr{S}(\bbrn)\bigcap H^{p_j}(\bbrn)$,
 where $X^p$ is defined as in \eqref{xpspacedef} and we take the convention $H^{\infty}=L^{\infty}$.
\end{proposition}
The proof of the proposition is quite standard by simply using Littlewood-Paley decomposition and classical dilation argument. 
For the sake of completeness, we provide the detailed proof in Section \ref{thm12pf}.

\hfill

Multi-variable extension of $A_p$ classes was introduced by Lerner, Ombrosi, P\'erez, Torres, and Trujillo-Gonz\'alez \cite{Le_Om_Pe_To_Tr2009}.
%and we will consider a bi-variable one for our purpose.
 \begin{defn}\cite{Le_Om_Pe_To_Tr2009}
 Let $1< p_1,\dots,p_l<\infty$ and $1/p=1/p_1+\dots+1/p_l$.
 Then we define $\mathrm{A}_{\ppp}$, $\ppp=(p_1,\dots,p_l)$, to be the class of $l$-tuples of weights $\vec{w}:=(w_1,\dots,w_l)$ satisfying
 $$\sup_Q\bigg[ \Big( \frac{1}{|Q|}\int_Q v_{\vec{w}}(x)\; dx \Big)^{\frac{1}{p}}\prod_{j=1}^{l}\Big(\frac{1}{|Q|}\int_Q \big(w_j(x)\big)^{1-p_j'}\; dx \Big)^{\frac{1}{p_j}}\bigg]<\infty$$ 
 where $p_j'$ denotes the H\"older conjugate of $p_j$ and
 $$v_{\vec{w}}(x):=\prod_{j=1}^{l}\big( w_j(x) \big)^{\frac{p}{p_j}}.$$
 \end{defn}
As an analogue of  \eqref{chaweight}, it was verified in \cite{Le_Om_Pe_To_Tr2009} that
the class $\mathrm{A}_{\ppp}$ can be characterized by a maximal inequality. 
For this one, let us define
 the multi-sublinear Hardy-Littlewood maximal operator $\mathbf{M}$  by
 $$\mathbf{M}\big(f_1,\dots,f_l \big)(x):=\sup_{Q:x\in Q}\bigg( \frac{1}{|Q|^l}\int_{Q^l}\prod_{j=1}^{l}\big|f_j(u_j) \big|\; d\uuu \bigg)$$
 for locally integrable functions $f_1,\dots,f_l$ on $\bbrn$, where $Q^l:=Q\times \cdots \times Q$, $d\uuu:=du_1\cdots du_l$, and
 the supremum is taken over all cubes in $\bbrn$ containing $x$.
   Of course, this maximal function is less than the product of maximal functions $\mathrm{M}f_j(x)$.

  The following theorem presents a maximal function characterization of the class $A_{\ppp}$.
   \begin{customthm}{F}\cite{Le_Om_Pe_To_Tr2009}\label{keylemma}
Let $1<p_1,\dots,p_l<\infty$ with $1/p=1/p_1+\dots+1/p_l$.
Then the inequality
 $$\big\Vert \mathbf{M}(f_1,\dots,f_l)\big\Vert_{L^p(v_{\vec{w}})}\lesssim \prod_{j=1}^{l}\Vert f_j\Vert_{L^{p_j}(w_j)}$$
 holds for all locally integrable functions $f_1,\dots,f_l$ if and only if
 $$\vec{w}=(w_1,\dots,w_l)\in \mathrm{A}_{\ppp}.$$
 \end{customthm}

 We now extend Theorem \ref{firstmainthm} to multilinear pseudo-differential operators in $\mathrm{Op}\mathbb{M}_lS_{\rho,\rho}^{m}(\bbrn)$.
 For this one, let us define the following $L^r$ version of $\mathbf{M}$ by
% similar to \eqref{Lrgen},
 $$\mathbf{M}_r(f_1,\dots,f_l)(x):=\big( \mathbf{M}\big(|f_1|^r,\dots, |f_l|^r \big)(x) \big)^{\frac{1}{r}}.$$
 Then we have the following multi-variable variant of Theorem \ref{firstmainthm}, which is the second main result of this paper.

  \begin{theorem}\label{mainpointest}
 Let $0< \rho<1$, $1< r\le 2$, and $m=-\frac{nl}{r}(1-\rho)$.
 Then every $\sigma\in \mathbb{M}_lS_{\rho,\rho}^{m}(\bbrn)$ satisfies
 \begin{equation}\label{mshrltsi}
 \mathcal{M}^{\sharp}_{{r}/{l}}\big( T_{\sigma}(f_1,\dots,f_l)\big)(x)\lesssim \mathbf{M}_r\big(f_1,\dots,f_l\big)(x), \q \q ~x\in\bbrn
 \end{equation}
for all $f_1,\dots,f_l\in\mathscr{S}(\bbrn)$.
 \end{theorem}
 We note that the inequality \eqref{mshrltsi} improves and generalizes the bilinear pointwise estimate in \cite{Na2015} that 
$$\mathscr{M}^{\sharp}\big(T_{\sigma}(f_1,f_2)\big)(x)\lesssim \mathrm{M}_2f_1(x)\mathrm{M}_2f_2(x), \q x\in\bbrn$$
if $0<\rho<\frac{1}{2}$ and $\sigma\in \mathbb{M}_2S_{\rho,\rho}^{-n(1-\rho)}(\bbrn)$.
As a result of Theorem \ref{mainpointest}, we obtain the following $L^{\infty}\times\cdots\times L^{\infty}\to BMO$ estimate, which extends \cite[Theorem 1.1]{Na2015} to $0<\rho<1$ in the multilinear setting.
\begin{corollary}
Let $0<\rho<1$ and $m=-\frac{nl}{2}(1-\rho)$. Suppose that
$\sigma\in \mathbb{M}_lS_{\rho,\rho}^{m}(\bbrn)$.
Then we have
\begin{equation}\label{infbmoest}
\big\Vert T_{\sigma}(f_1,\dots,f_l)\big\Vert_{BMO}\lesssim \prod_{j=1}^{l}\Vert f_j\Vert_{L^{\infty}(\bbrn)}
\end{equation}
for all $f_1,\dots,f_l\in\mathscr{S}(\bbrn)$.
\end{corollary}
To achieve the inequality \eqref{infbmoest}, we recall the $BMO$ characterization
$$
\|f\|_{BMO} \sim_p \sup_{Q}\inf_{c_Q}
\left( \frac{1}{|Q|}\int_Q |f(x)-c_Q|^p\, dx \right)^{\frac{1}{p}}, \qq 0<p<\infty.
$$
See \cite[page 517]{Str1979} especially for $0<p<1$. 
By applying the above norm equivalence, we have
$$  \big\Vert T_{\sigma}(f_1,\dots,f_l)\big\Vert_{BMO} \lesssim   \big\Vert  \mathcal{M}^{\sharp}_{{2}/{l}}\big( T_{\sigma}(f_1,\dots,f_l)\big)    \big\Vert_{L^{\infty}(\bbrn)}$$
and then this can be further estimated, via Theorem \ref{mainpointest} and the $L^{\infty}$ boundedness of $\mathrm{M}_2$, by a constant times
$$\big\Vert  \mathbf{M}_2(f_1,\dots,f_l)   \big\Vert_{L^{\infty}(\bbrn)}\le \big\Vert \mathrm{M}_2f_1\cdots \mathrm{M}_2f_l\big\Vert_{L^{\infty}(\bbrn)}\le \prod_{j=1}^{l}\Vert f_j\Vert_{L^{\infty}(\bbrn)},$$
as desired.

 \hfill

 As an application of Theorem \ref{mainpointest}, we obtain weighted estimates for multilinear pseudo-differential operators in $\mathrm{Op}\mathbb{M}_lS_{\rho,\rho}^{m}(\bbrn)$.  
  \begin{theorem}\label{mainthm}
 Let $1< r\le 2$ and $r<p_1,\dots,p_l<\infty$ with $1/p=1/p_1+\dots+1/p_l$. 
 Suppose that $0< \rho<1$ and $\sigma\in \mathbb{M}_lS_{\rho,\rho}^{-\frac{nl}{r}(1-\rho)}(\bbrn)$.
If an $l$-tuple of weights $\vec{w}=(w_1,\dots, w_l)$ belongs to the class $A_{\frac{p_1}{r},\dots,\frac{p_l}{r}}$, then the weighted norm inequality
$$\big\Vert T_{\sigma}(f_1,\dots,f_l)\big\Vert_{L^p(v_{\vec{w}})}\lesssim \prod_{j=1}^{l}\Vert f_j\Vert_{L^{p_j}(w_j)}$$
holds 
for $f_1,\dots,f_l\in \mathscr{S}(\bbrn)$.
 \end{theorem}

 \hfill
 
{\bf Notation.}
Let $\bbn$ be the set of all natural numbers and $\bbn_0:=\bbn\cup \{0\}$.
We denote by $\chi_U$ the characteristic function of a set $U$.
For each cube $Q$ in $\bbrn$, let $\ell(Q)$ denote the side-length of $Q$.

 \hfill
 
 {\bf Organization.}
The linear operator will be studied only in Section \ref{linearproof}, including the proof of Theorem \ref{firstmainthm}.
 In the remaining sections we focus on multilinear operators.
Indeed, in Section \ref{decompsection} we set up the decomposition of multilinear pseudo-differential operators on which the proofs of Proposition \ref{multilineargerho} and Theorem \ref{mainpointest} are based.  The proof of Proposition \ref{multilineargerho} will be given in Section \ref{thm12pf}. We present some important multilinear estimates which will play a crucial role in the proof of Theorem \ref{mainpointest} in Section \ref{mefmosec} and by using those we completely prove Theorem \ref{mainpointest} in Section \ref{thm14pf}. The last section is devoted to the proof of Theorem \ref{mainthm}.

 \section{Linear Estimates : Proof of Theorem \ref{firstmainthm}}\label{linearproof}
 
 Let $\phi$ be a Schwartz function on $\bbrn$ such that its Fourier transform $\wh{\phi}$ is equal to $1$ on the unit ball centered at the origin and is supported in the ball of radius $2$.
 Let $\psi\in\mathscr{S}(\bbrn)$ satisfy $\wh{\psi}(\xi):=\wh{\phi}(\xi)-\wh{\phi}(2\xi)$ for $\xi\in \bbrn$.
 For each $k\in\bbn$, we define $\psi_k(x):=2^{kn}\psi(2^k x)$ for $x\in \bbrn$. Then $\{\phi\}\cup \{\psi_k\}_{k\in\bbn}$ forms inhomogeneous Littlewood-Paley partition of unity. Note that
 $$\supp(\wh{\psi_k})\subset \big\{\xi\in \bbrn : 2^{k-1}\le |\xi|\le 2^{k+1}\big\}, \q k\in\bbn$$
 and
 \begin{equation}\label{linearLPdecom}
 \wh{\phi}(\xi)+\sum_{k\in\bbn}\wh{\psi_k}(\xi)=1.
\end{equation}
Using \eqref{linearLPdecom}, we decompose  $\sigma \in S_{\rho,\rho}^{m}(\bbrn)$ as
$$\sigma(x,\xi)=\sigma(x,\xi)\wh{\phi}(\xi)+\sum_{k\in\bbn}\sigma(x,\xi)\wh{\psi_k}(\xi)=:\sigma_0(x,\xi)+\sum_{k\in\bbn}\sigma_k(x,\xi), \q x, \xi\in\bbrn$$
 and then express
 \begin{equation}\label{linsigmadecom}
 T_{\sigma}f=\sum_{k\in\bbn_0}T_{\sigma_k}f
 \end{equation}
 where $T_{\sigma_k}$ are the (linear) pseudo-differential operators associated with $\sigma_k\in S_{\rho,\rho}^{m}(\bbrn)$. We observe that for all multi-indices $\alpha,\beta\in (\bbn_0)^n$
 $$\big| \partial_x^{\alpha}\partial_{\xi}^{\beta}\sigma_k(x,\xi)\big|\lesssim_{\alpha,\beta}\big( 1+|\xi|\big)^{m-\rho(|\beta|-|\alpha|)}$$
 where the constant in the inequality is independent of $k$.
 Now we write
 $$T_{\sigma_k}f(y)=\int_{\bbrn}  K_k(y,y-u)f(u)    \; du$$
 where 
 \begin{equation}\label{kkyudef}
 K_k(y,u):=\int_{\bbrn} \sigma_k(y,\xi)e^{2\pi i\langle u,\xi\rangle}    \; d\xi.
 \end{equation}
 Then we first prove the following kernel estimate.
 \begin{lemma}\label{linearkeyest}
 Let $0\le \rho<1$ and $m\in\bbr$. Suppose that $\sigma\in S_{\rho,\rho}^m(\bbrn)$ and let $K_k$ be defined as in \eqref{kkyudef}.
 For arbitrary nonnegative $N\ge 0$ and $1\le r\le 2$, we have
 \begin{align*}
 \Big\Vert \big(1+2^{k\rho}|u|\big)^NK_k(y,u)\Big\Vert_{L^{r'}(u \in \bbrn)}&\lesssim_N 2^{k(m+\frac{n}{r})},\\
  \Big\Vert \big(1+2^{k\rho}|u|\big)^N\nabla_{y}K_k(y,u)\Big\Vert_{ L^{r'}(u \in \bbrn)}&\lesssim_N 2^{k(\rho+m+\frac{n}{r})},\\
   \Big\Vert \big(1+2^{k\rho}|u|\big)^N\nabla_{u}K_k(y,u)\Big\Vert_{ L^{r'}(u \in \bbrn)}&\lesssim_N 2^{k(1+m+\frac{n}{r})}
 \end{align*}
 uniformly in $y\in\bbrn$, where $\nabla_y$ and $\nabla_u$ denote the gradient operators with respect to $y$ and $u$, respectively.
 \end{lemma}
 \begin{proof}
  The bilinear case with $r=2$ is already discussed in \cite[page 2757]{Mi_To2020} and the proof is almost the same by using the Hausdorff-Young inequality. 
 Indeed, since $\sigma_k\in S_{\rho,\rho}^{m}(\bbrn)$,
 \begin{equation}\label{paxibesikest1}
 \big|\partial_{\xi}^{\beta}\sigma_k(y,\xi) \big|\lesssim 2^{k(m-\rho|\beta|)}\chi_{\{ 1+|\xi|\sim 2^k\}}.
 \end{equation}
 Then the Hausdorff-Young inequality yields that
 \begin{align*}
 \big\Vert \big(2^{k\rho}u  \big)^{\beta}K_k(y,u)\big\Vert_{L^{r'}(u \in \bbrn)}&\lesssim 2^{k\rho |\beta|}\big\Vert \partial_{\xi}^{\beta}\sigma_k(y,\xi)\big\Vert_{ L^{r}(\xi \in \bbrn)}  \\ &\sim 2^{k\rho |\beta|} 2^{k(m-\rho |\beta|)}2^{\frac{kn}{r}}=2^{k(m+\frac{n}{r})}
 \end{align*}
 from which the first assertion follows. The remaining inequalities can be also derived from the estimates
\begin{equation}\label{paxibesikest2}
 \begin{aligned}
  \big|\partial_{\xi}^{\beta}\nabla_y\sigma_k(y,\xi) \big|&\lesssim 2^{k(m+\rho-\rho |\beta|)}\chi_{\{1+|\xi| \sim 2^k\}}\\
    \big|\partial_{\xi}^{\beta} \big( \xi \,\sigma_k(y,\xi)\big) \big|&\lesssim 2^{k(m+1-\rho |\beta|)}\chi_{\{1+|\xi| \sim 2^k\}}
 \end{aligned}
 \end{equation}
 in the same way.
 \end{proof}

 \subsection{Key Estimates}
 
 Let $P$ be the concentric dilate of $Q$ with $\ell(P)\ge 10\sqrt{n} \ell(Q)$ and $\chi_P$ denote the characteristic function of $P$.
 Now we divide 
 \begin{equation*}
 f= f\chi_P+f\chi_{P^c}=:f^{(0)}+f^{(1)} 
 \end{equation*}
 so that $$T_{\sigma_k}f=T_{\sigma_k}f^{(0)}+T_{\sigma_k}f^{(1)}.$$
Then the following propositions are key estimates which will be repeatedly used in the proof of Theorem \ref{firstmainthm}.
 \begin{proposition}\label{linearkeypropo1}
Let $0< \rho<1$, $1< r\le 2$, and $m=-\frac{n}{r}(1-\rho)$. Suppose that $x\in Q$.
\begin{enumerate}
\item Suppose that $0< \rho< \frac{r}{2}$ and $k\in\bbn_0$.
Then every $\sigma\in S_{\rho,\rho}^{m}(\bbrn)$ satisfies
 $$\bigg(\frac{1}{|Q|}\int_Q \big| T_{\sigma_k}f^{(0)}(y)\big|^r \; dy\bigg)^{\frac{1}{r}}\lesssim \Big(\frac{\ell(P)}{\ell(Q)^{\rho}} \Big)^{\frac{ n}{r}}\mathrm{M}_rf(x).$$

\item 
Suppose that $\frac{r}{2}\le \rho <1$, $\frac{2\rho -r}{2-r}<\lambda<\rho$, and $k\in\bbn_0$.
 Then every $\sigma\in S_{\rho,\rho}^{m}(\bbrn)$ satisfies
 $$\bigg(\frac{1}{|Q|}\int_Q \big| T_{\sigma_k}f^{(0)}(y)\big|^r \; dy \bigg)^{\frac{1}{r}}\lesssim \big( 2^k \ell(Q)\big)^{\lambda n \frac{1-\rho}{r(1-\lambda)} }\Big(\frac{\ell(P)}{\ell(Q)^{\rho}} \Big)^{\frac{n}{r}}\mathrm{M}_rf(x).$$
 \end{enumerate}
 \end{proposition}

 \begin{proposition}\label{linearkeypropo2}
 Let $0\le \rho<1$, $1\le r\le 2$, $m=-\frac{n}{r}(1-\rho)$, and $k\in\bbn_0$. Suppose that $x,y\in Q$.
 Then $\sigma\in S_{\rho,\rho}^{m}(\bbrn)$ satisfies
 \begin{equation}\label{linearkeypropo2_1}
 \big| T_{\sigma_k}f^{(1)}(y)\big|\lesssim_N \big( 2^{k\rho} \ell(P)\big)^{-(N-\frac{n}{r})}\mathrm{M}_rf(x)
  \end{equation}
and
\begin{align}\label{linearkeypropo2_2}
\big| T_{\sigma_k}f^{(1)}(y)-T_{\sigma_k}f^{(1)}(x)\big| \lesssim_N 2^k\ell(Q)\big( 2^{k\rho} \ell(P)\big)^{-(N-\frac{n}{r})}\mathrm{M}_rf(x)
\end{align}
for any $N>\frac{n}{r}$.
\end{proposition}

 Before proving the above propositions, we introduce a useful boundedness property for $T_{\sigma}$ on Lebesgue spaces which will be extended to the multilinear setting in Section \ref{mefmosec} as Lemma \ref{Naibopropo}.
   \begin{lemma}\label{MiYapropo}\cite[Lemma 5.6]{Mi_Ya1987}
  Let $1< r\le 2$, $0< \rho\le \frac{r}{2}$, $\rho<1$, and $m=-\frac{n}{r}(1-\rho)$.
 Then every $\sigma\in S_{\rho,\rho}^m(\bbrn)$ satisfies
 $$\big\Vert T_{\sigma}f \big\Vert_{L^{\frac{r}{\rho}}(\bbrn)}\lesssim \Vert f\Vert_{L^r(\bbrn)}$$
 for $f\in \mathscr{S}(\bbrn)$.
 \end{lemma}

 The first estimate of Proposition \ref{linearkeypropo1} will be a consequence of Lemma \ref{MiYapropo}, but it cannot be directly applied for the second assertion in which the condition $\frac{r}{2}\le \rho<1$ is assumed. So we introduce alternative estimates applicable to the case when $\frac{r}{2}\le \rho<1$.
  \begin{lemma}\label{linearrhoge12}
 Let $1< r<2$, $\frac{r}{2} \le \rho<1$, and $m=-\frac{n}{r}(1-\rho)$.
 Suppose that $k\in\bbn_0$ and
 \begin{equation}\label{2rhomr2mr}
 \frac{2\rho -r}{2-r}<\lambda<\rho.
 \end{equation}
 Then every $\sigma\in S_{\rho,\rho}^{m}(\bbrn)$ satisfies
 \begin{equation}\label{linearrhoge12est}
 \big\Vert T_{\sigma_k}f \big\Vert_{L^{\frac{r(1-\lambda)}{(\rho-\lambda)}}(\bbrn)}\lesssim 2^{\lambda n k \frac{1-\rho}{r(1-\lambda)} }\Vert f\Vert_{L^r(\bbrn)}
 \end{equation}
 for $f\in \mathscr{S}(\bbrn)$.
 \end{lemma}
 \begin{proof}
 Let $\frac{2\rho -r}{2-r}<\lambda<\rho$.
We define
$$c_k(x,\xi):=\sigma_k(2^{-\lambda k}x, 2^{\lambda k}\xi).$$
 Then for any multi-indices $\alpha,\beta \in (\bbn_0)^n$, we have
  \begin{equation}\label{linearckins}
 \big| \partial_x^{\alpha}\partial_{\xi}^{\beta}c_k(x,\xi)     \big|\lesssim \big( 1+|\xi| \big)^{\frac{m}{1-\lambda}-\frac{\rho-\lambda}{1-\lambda}(|\beta|-|\alpha|)}
 \end{equation} 
 where the constant in the inequality is independent of the index $k$. That is,
 \begin{equation*}
 c_k\in S_{\frac{\rho-\lambda}{1-\lambda},\frac{\rho-\lambda}{1-\lambda}}^{\frac{m}{1-\lambda}}(\bbrn)\q \text{ uniformly in }~k.
 \end{equation*} 
 The proof of \eqref{linearckins} is a straight-forward computation, but we provide it in the appendix for the sake of completeness.
We also observe that 
$$0<\frac{\rho-\lambda}{1-\lambda}<\frac{r}{2} \q \text{ and }\q \frac{m}{1-\lambda}=-\frac{n}{r}\Big(1-\frac{\rho-\lambda}{1-\lambda}\Big)$$
in view of \eqref{2rhomr2mr},
 and thus Lemma \ref{MiYapropo} yields that
 \begin{equation}\label{linearckest}
 \big\Vert T_{c_k}f\big\Vert_{L^{\frac{r(1-\lambda)}{\rho-\lambda}}(\bbrn)}\lesssim \Vert f\Vert_{L^r(\bbrn)}\q \text{uniformly in }~k.
 \end{equation}
Since
$$T_{\sigma_k}f(x)=T_{c_k}\big( f(2^{-\lambda k}\cdot)\big)(2^{\lambda k}x),$$
 we apply \eqref{linearckest} to obtain
 \begin{align*}
 \big\Vert T_{\sigma_k}f\big\Vert_{L^{\frac{r(1-\lambda)}{\rho-\lambda}}(\bbrn)}&=2^{-\lambda nk \frac{\rho-\lambda}{r(1-\lambda)}}\big\Vert T_{c_k}\big(f(2^{-\lambda k}\cdot) \big)\big\Vert_{L^{\frac{r(1-\lambda)}{\rho-\lambda}}(\bbrn)}\\
 &\lesssim 2^{-\lambda nk \frac{\rho-\lambda}{r(1-\lambda)}}\big\Vert f(2^{-\lambda k}\cdot)\big\Vert_{L^r(\bbrn)}\\
 &= 2^{-\lambda nk \frac{\rho-\lambda}{r(1-\lambda)}} 2^{\lambda n k\frac{1}{r}}\Vert f\Vert_{L^r(\bbrn)}\\
 &=2^{\lambda n k \frac{1-\rho}{r(1-\lambda)} } \Vert f\Vert_{L^r(\bbrn)},
 \end{align*}
 as desired.
 \end{proof}

 Now we prove Propositions \ref{linearkeypropo1} and \ref{linearkeypropo2}.
 \begin{proof}[Proof of Proposition \ref{linearkeypropo1}]
Let us first consider the case when $0<\rho<\frac{r}{2}$.
 By H\"older's inequality and Lemma \ref{MiYapropo},
 \begin{align*}
 \bigg(\frac{1}{|Q|}\int_Q \big| T_{\sigma_k}f^{(0)}(y)\big|^r\;dy\bigg)^{\frac{1}{r}}&\le \frac{1}{|Q|^{\frac{\rho}{r}}}\big\Vert T_{\sigma_k}f^{(0)}\big\Vert_{L^{\frac{r}{\rho}}(\bbrn)}\\
 &\lesssim \frac{1}{\ell(Q)^{\frac{n\rho  }{r}}}\Vert f\Vert_{L^r(P)}\le \Big( \frac{\ell(P)}{\ell(Q)^{\rho}}\Big)^{\frac{n}{r}}\mathrm{M}_rf(x)
 \end{align*}
 as $x\in Q\subset P$. This proves the first assertion.
 
Now assume $\frac{r}{2}\le \rho<1$. In this case, we apply H\"older's inequality and Lemma \ref{linearrhoge12} to deduce
  \begin{align*}
\bigg( \frac{1}{|Q|}\int_Q \big| T_{\sigma_k}f^{(0)}(y)\big|^r\;dy \bigg)^{\frac{1}{r}}&\le \frac{1}{|Q|^{\frac{\rho-\lambda}{r(1-\lambda)}}}\big\Vert T_{\sigma_k}f^{(0)}\big\Vert_{L^{\frac{r(1-\lambda)}{\rho-\lambda}}(\bbrn)}\\
 &\lesssim \frac{1}{\ell(Q)^{n \frac{\rho-\lambda}{r(1-\lambda)}}}2^{\lambda n k \frac{1-\rho}{r(1-\lambda)} } \Vert f\Vert_{L^r(P)}\\
 &\lesssim 2^{\lambda n k \frac{1-\rho}{r(1-\lambda)} }\ell(Q)^{-n \frac{\rho-\lambda}{r(1-\lambda)} }\ell(P)^{\frac{n}{r}}\mathrm{M}_rf(x)\\
 &=\big( 2^k \ell(Q)\big)^{\lambda n \frac{1-\rho}{r(1-\lambda)} }\Big(\frac{\ell(P)}{\ell(Q)^{\rho}} \Big)^{\frac{n}{r}}\mathrm{M}_rf(x),
 \end{align*}
 which completes the proof.
 \end{proof}

  \begin{proof}[Proof of Proposition \ref{linearkeypropo2}]
  
We first consider \eqref{linearkeypropo2_1}. Let $N>\frac{n}{r}$ and $y\in Q$.
 Using H\"older's inequality,
 \begin{align*}
 \big| T_{\sigma_k}f^{(1)}(y) \big|&\le \int_{ P^c}\big|K_k(y,y-u) \big| \big| f(u)\big|\; du\\
 &\le \bigg\Vert \ell(P)^{\frac{n}{r}}\Big( \frac{|\cdot|}{\ell(P)}\Big)^N K_k(y,\cdot)\bigg\Vert_{L^{r'}(\bbrn  )} \bigg\Vert  \ell(P)^{-\frac{n}{r}}\bigg( \frac{|y-\cdot|}{\ell(P)}\bigg)^{-N} f \bigg\Vert_{L^r(P^c)}.
 \end{align*}
 The $L^{r'}$-norm can be estimated as
 \begin{align}\label{linearlpvnkkest}
\bigg\Vert \ell(P)^{\frac{n}{r}}\Big( \frac{|\cdot|}{\ell(P)}\Big)^N K_k(y,\cdot)\bigg\Vert_{L^{r'}(\bbrn  )}&=\ell(P)^{-(N-\frac{n}{r})}\Big\Vert |\cdot|^N K_k(y,\cdot) \Big\Vert_{L^{r'}(\bbrn)}\nonumber\\
&\le \ell(P)^{-(N-\frac{n}{r})}2^{-k\rho N} \Big\Vert \big( 1+2^{k\rho}|\cdot|\big)^N K_k(y,\cdot) \Big\Vert_{L^{r'}(\bbrn)}\nonumber\\
 &\lesssim \ell(P)^{-(N-\frac{n}{r})}2^{-k\rho N}2^{k(m+\frac{n}{r})}\nonumber\\
 &=\big(2^{k\rho}\ell(P)\big)^{-(N-\frac{n}{r})}
 \end{align}
 where the second inequality follows from Lemma \ref{linearkeyest}.
For the  $L^r$-norm, we note that for $x,y\in Q$, and $u\in P^c$,
\begin{equation*}
|y-u| \ge |x-u|-|x-y|\gtrsim_{n}\ell(P)+|x-u|
\end{equation*}
 and thus for $N>\frac{n}{r}$
 \begin{equation}\label{linearlpyvnest}
  \bigg\Vert  \ell(P)^{-\frac{n}{r}}\Big( \frac{|y-\cdot|}{\ell(P)}\Big)^{-N} f \bigg\Vert_{L^r(P^c)}\lesssim \bigg( \int_{\bbrn}\frac{1}{\ell(P)^{n}}\frac{1}{\big(1+\frac{|x-u|}{\ell(P)} \big)^{rN}} \big|f(u)\big|^r      \; du\bigg)^{1/r}\lesssim \mathrm{M}_rf(x).
 \end{equation}
Finally, we have
\begin{equation*}
 \big| T_{\sigma_k}f^{(1)}(y) \big| \lesssim \big(2^{k\rho}\ell(P)\big)^{-(N-\frac{n}{r})}\mathrm{M}_rf(x),
 \end{equation*}
as desired.

We next consider \eqref{linearkeypropo2_2}.
 The structure of the proof is very similar to that of \eqref{linearkeypropo2_1}, by replacing the kernel $K_k(y,y-u)$ by
 $$H_k(y,x,u):=K_k(y,y-u)-K_k(x,x-u).$$
 Assume $x,y\in Q$ and $N>\frac{n}{r}$.
 Using H\"older's inequality, we have
 \begin{align*}
 &\big| T_{\sigma_k}f^{(1)}(y)-T_{\sigma_k}f^{(1)}(x)\big|\\
 &\le \int_{ P^c}\big| H_k(y,x,u)\big|\big| f(u)\big|\; du\\
 &\le \bigg\Vert  \ell(P)^{\frac{n}{r}}\Big( \frac{|y-\cdot|}{\ell(P)}\Big)^N H_k(y,x,\cdot)\bigg\Vert_{L^{r'}( P^c)} \bigg\Vert \ell(P)^{-\frac{n}{r}}\Big( \frac{|y-\cdot|}{\ell(P)}\Big)^{-N}f\bigg\Vert_{L^r( P^c)}.
 \end{align*}
 
  Then the following estimate holds as an analogue of \eqref{linearlpvnkkest}: for $x,y\in Q$,
 \begin{align*}
 &\bigg\Vert  \ell(P)^{\frac{n}{r}}\Big( \frac{|y-\cdot|}{\ell(P)}\Big)^N H_k(y,x,\cdot)\bigg\Vert_{L^{r'}( P^c)} \\
 &\lesssim \bigg\Vert \ell(Q) \ell(P)^{\frac{n}{r}} \Big( \frac{|y-\cdot|}{\ell(P)}\Big)^N \int_0^1 \big| \nabla K_k\big(y(t),y(t)-\cdot \big)\big|\; dt      \bigg\Vert_{L^{r'}(P^c)}\nonumber\\
 &\lesssim \ell(Q)\ell(P)^{-(N-\frac{n}{r})}\int_0^1 \Big\Vert   |y(t)-\cdot|^N  \big|\nabla K_k\big(y(t),y(t)-\cdot\big) \big|      \Big\Vert_{L^{r'}(\bbrn)} \; dt\nonumber\\
 & \le \ell(Q)\ell(P)^{-(N-\frac{n}{r})}2^{-k\rho N}\int_0^1 \Big\Vert   \big(1+2^{k\rho}|\cdot|\big)^N \big|\nabla K_k\big(y(t),\cdot \big) \big|      \Big\Vert_{L^{r'}(\bbrn)} \; dt\\
 &\lesssim     \ell(Q)\ell(P)^{-(N-\frac{n}{r})}2^{-k\rho N}2^{k(1+m+\frac{n}{r})}      =2^k\ell(Q) \big( 2^{k\rho}\ell(P)\big)^{-(N-\frac{n}{r})}
 \end{align*} where $y(t):=ty+(1-t)x\in Q$ so that $|y-u|\lesssim |y(t)-u|$ for $u\in P^c$.
 Here, the last inequality holds due to Lemma \ref{linearkeyest}.
This, together with \eqref{linearlpyvnest}, yields that
\begin{equation*}
\big| T_{\sigma_k}f^{(1)}(y)-T_{\sigma_k}f^{(1)}(x)\big|\lesssim  2^k\ell(Q) \big( 2^{k\rho}\ell(P)\big)^{-(N-\frac{n}{r})}\mathrm{M}_rf(x).
\end{equation*}
This finishes the proof.
 \end{proof}

 \hfill
 
 \subsection{Proof of Theorem \ref{firstmainthm}}

 Let $1<r\le 2$. We fix a cube $Q$ and a point $x\in Q$. Then it suffices to show that
 \begin{equation}\label{linearmest0}
\bigg( \frac{1}{|Q|}\int_Q \big| T_{\sigma}f(y)\big|^r\; dy\bigg)^{\frac{1}{r}} \lesssim \mathrm{M}_rf(x) \q \text{if}~\ell(Q)\ge 1
 \end{equation}
 and 
 \begin{equation}\label{linearmest1}
 \inf_{c_Q\in\bbc}\bigg(\frac{1}{|Q|}\int_Q \big|T_{\sigma}f(y) -c_Q\big|^r\; dy \bigg)^{\frac{1}{r}}\lesssim \mathrm{M}_rf(x) \q \text{if}~ \ell(Q)<1
 \end{equation}
 uniformly in $Q$ and $x$.
 
 \subsubsection{Proof of \eqref{linearmest0}}
 Assume $\ell(Q)\ge 1$ and let $Q^*$ be the dilate of $Q$ by a factor of $10\sqrt{n}$.
 Then we write
 \begin{equation}\label{lineardildecfg}
 f=f\chi_{Q^*}+f\chi_{(Q^*)^c}=:\mathbf{f}^{(0)}+\mathbf{f}^{(1)}
 \end{equation}
 and thus 
$$\bigg( \frac{1}{|Q|}\int_Q \big|T_{\sigma}f(y)\big|^r \; dy \bigg)^{\frac{1}{r}}\le \bigg(\frac{1}{|Q|}\int_Q\big| T_{\sigma}\mathbf{f}^{(0)}(y)\big|^r\; dy\bigg)^{\frac{1}{r}}+\bigg(\frac{1}{|Q|}\int_Q \big| T_{\sigma}\mathbf{f}^{(1)}(y)\big|^r\; dy\bigg)^{\frac{1}{r}}.$$
 
 By the $L^r$ boundedness for $T_{\sigma}$ in Theorem \ref{thma} with
 $$m=-\frac{n}{r}(1-\rho)<-n(1-\rho)\Big(\frac{1}{r}-\frac{1}{2}\Big),$$
 we first obtain 
 \begin{equation*}
 \bigg( \frac{1}{|Q|}\int_Q\big| T_{\sigma}\mathbf{f}^{(0)}(y)\big|^r\; dy \bigg)^{\frac{1}{r}}\le \frac{1}{|Q|^{\frac{1}{r}}}\big\Vert    T_{\sigma}\mathbf{f}^{(0)}     \big\Vert_{L^r(\bbrn)}\lesssim \frac{1}{\ell(Q)^{\frac{n}{r}}}\Vert f\Vert_{L^r(Q^*)}\lesssim \mathrm{M}_rf(x).
 \end{equation*}
 
 To estimate the other term, using the decomposition \eqref{linsigmadecom}, we express
 $$T_{\sigma}\mathbf{f}^{(1)}=\sum_{k\in\bbn_0}T_{\sigma_k}\mathbf{f}^{(1)}.$$
 Then choosing $N>\frac{n}{r}$ and using \eqref{linearkeypropo2_1},
$$  \big|T_{\sigma_k}\mathbf{f}^{(1)}(y)\big|\lesssim \big( 2^{\rho k}\ell(Q)\big)^{-(N-\frac{n}{r})} \mathrm{M}_rf(x),$$
 which proves
 $$\big| T_{\sigma}\mathbf{f}^{(1)}(y)\big|\lesssim \mathrm{M}_rf(x) \sum_{k\in\bbn_0}  \big( 2^{\rho k}\ell(Q)\big)^{-(N-\frac{n}{r})} \lesssim \mathrm{M}_rf(x)  $$
 as $N>\frac{n}{r}$ and $\ell(Q)\ge 1$.
 This completes the proof of \eqref{linearmest0}.
 
 \hfill

 \subsubsection{Proof of \eqref{linearmest1}}
 Assume $\ell(Q)<1$. In this case we need to consider separately the cases $0<\rho< \frac{r}{2}$ and $\frac{r}{2}\le \rho<1$, in which different decompositions will be performed.\\
 
 {\bf The case when $0<\rho< \frac{r}{2}$}
 
 Let $Q_{\rho}$ be the concentric dilate of $Q$ with $\ell(Q_{\rho})=10\sqrt{n}  \ell(Q)^{\rho}$. We split $f$ as 
 \begin{equation}\label{lineardecomfqrho}
f=f\chi_{Q_{\rho}}+f\chi_{(Q_{\rho})^c}=:\mathtt{f}^{(0)}+\mathtt{f}^{(1)}.
\end{equation}
Then the left-hand side of \eqref{linearmest1} is less than the sum of
\begin{equation*}
 \mathcal{I}_1:=\bigg(\frac{1}{|Q|}\int_Q \big| T_{\sigma}\mathtt{f}^{(0)}(y)\big|^r\; dy\bigg)^{\frac{1}{r}}
\end{equation*}
and
\begin{equation*}
\mathcal{I}_2:=\inf_{c_Q\in\mathbb{C}}\bigg(\frac{1}{|Q|}\int_Q  \big|T_{\sigma}\mathtt{f}^{(1)}(y)-c_Q\big|^r \; dy\bigg)^{\frac{1}{r}}.
\end{equation*}
 
First of all, by using H\"older's inequality and Lemma \ref{MiYapropo}, it can be easily verified that
 \begin{align*}
 \mathcal{I}_1&\le \frac{1}{|Q|^{\frac{\rho}{r}}}\big\Vert    T_{\sigma}\mathtt{f}^{(0)}  \big\Vert_{L^{\frac{r}{\rho}}(\bbrn)}\lesssim \frac{1}{\ell(Q)^{\frac{n\rho }{r}}}\Vert f\Vert_{L^r(Q_{\rho})}\lesssim \mathrm{M}_rf(x).
 \end{align*}

 To estimate $\mathcal{I}_2$, using the decomposition \eqref{linsigmadecom} and taking
  \begin{equation*}
 c_Q=\sum_{k:2^k\ell(Q)<1}T_{\sigma_k}\mathtt{f}^{(1)}(x),
 \end{equation*}
 we write
  \begin{equation*}
 \big|T_{\sigma} \mathtt{f}^{(1)}(y)-c_Q\big|\le \sum_{k:2^k\ell(Q)\ge 1}\big| T_{\sigma_k}\mathtt{f}^{(1)}(y)   \big|+\sum_{k:2^k\ell(Q)<1}\big| T_{\sigma_k}\mathtt{f}^{(1)}(y) -T_{\sigma_k} \mathtt{f}^{(1)}(x)    \big|.
 \end{equation*}
 We first see that
 \begin{equation*}
 \big| T_{\sigma_k}\mathtt{f}^{(1)}(y)   \big|\lesssim \big( 2^k\ell(Q)\big)^{-\rho(N-\frac{n}{r})}\mathrm{M}_rf(x)
 \end{equation*}
 by applying \eqref{linearkeypropo2_1}, and this proves
 \begin{equation}\label{linearsk2klqge1}
 \sum_{k:2^k\ell(Q)\ge 1}\big| T_{\sigma_k}\mathtt{f}^{(1)}(y)   \big|\lesssim  \mathrm{M}_rf(x)\sum_{k:2^k\ell(Q)\ge 1}\big( 2^k\ell(Q)\big)^{-\rho(N-\frac{n}{r})}\lesssim \mathrm{M}_rf(x)
 \end{equation} as $\rho>0$.
 For the remaining part, we pick a number $N$ such that 
 $$\frac{n}{r}<N<\frac{n}{r}+\frac{1}{\rho}$$
 and then apply \eqref{linearkeypropo2_2} to obtain
 \begin{equation}\label{linearsk2klqge2}
\sum_{k:2^k\ell(Q)<1}\big| T_{\sigma_k}\mathtt{f}^{(1)}(y) -T_{\sigma_k} \mathtt{f}^{(1)}(x)    \big|\lesssim \sum_{k:2^k\ell(Q)<1}  \big( 2^k\ell(Q)\big)^{1-\rho(N-\frac{n}{r})} \mathrm{M}_rf(x) \sim \mathrm{M}_rf(x)  
 \end{equation}
 as $1-\rho(N-\frac{n}{r})>0$.
 The estimates  \eqref{linearsk2klqge1} and \eqref{linearsk2klqge2}  provide the required conclusion for $\mathcal{I}_2$:
 $$\mathcal{I}_2\lesssim \mathrm{M}_rf(x).$$
 Hence \eqref{linearmest1} holds for $0< \rho<\frac{r}{2}$.

 \hfill
 
 {\bf The case when $\frac{r}{2}\le \rho<1$}

 By using the decomposition \eqref{linsigmadecom},
 the left-hand side of \eqref{linearmest1} is bounded by the sum of
 \begin{align*}
 \mathcal{J}_1&:=\inf_{c_Q\in \bbc}\bigg(\frac{1}{|Q|}\int_Q \Big| \sum_{k:2^k\ell(Q)< 1}T_{\sigma_k}f(y)-c_{Q}\Big|^r \; dy\bigg)^{\frac{1}{r}},\\
 \mathcal{J}_2&:=\sum_{\substack{k:2^k\ell(Q)\ge 1,\\~ 2^{\rho k       }\ell(Q)< 1}} \bigg( \frac{1}{|Q|}\int_Q \big| T_{\sigma_k}f(y)\big|^r\;dy\bigg)^{\frac{1}{r}},\\
 \mathcal{J}_3&:=\sum_{k:2^{\rho k}\ell(Q)\ge 1}\bigg(\frac{1}{|Q|}\int_Q \big| T_{\sigma_k}f(y)\big|^r\;dy\bigg)^{\frac{1}{r}}.
 \end{align*}
Now the proof of \eqref{linearmest1} is reduced to the estimates
\begin{equation}\label{lineariiestm2}
\mathcal{J}_{\mathrm{i}}\lesssim \mathrm{M}_rf(x),\q \mathrm{i}=1,2,3,
\end{equation}
which will be proved below.
For the proof of \eqref{lineariiestm2}, we set $\frac{2\rho -r}{2-r}<\lambda<\rho$  so that the inequality \eqref{linearrhoge12est} holds for $\frac{r}{2}\le \rho<1$.\\
 
To estimate $\mathcal{J}_1$, 
recalling $Q_{\rho}$ is the concentric dilate of $Q$ whose side-length is $10\sqrt{n} \ell(Q)^{\rho}$ and using \eqref{lineardecomfqrho},
we write 
 \begin{align*}
T_{\sigma_k}f&=T_{\sigma_k}\mathtt{f}^{(0)}+T_{\sigma_k}\mathtt{f}^{(1)}.
 \end{align*}
 Then by taking
 $$c_Q=\sum_{k:2^k\ell(Q)<1}T_{\sigma_k}\mathtt{f}^{(1)}(x),$$
 we bound $ \mathcal{J}_1$ by
 \begin{align*}
\sum_{k:2^k\ell(Q)<1}\left[ \bigg( \frac{1}{|Q|}\int_Q \big|T_{\sigma_k}\mathtt{f}^{(0)}(y)\big|^r \; dy\bigg)^{\frac{1}{r}}+\bigg(\frac{1}{|Q|}\int_Q \big|T_{\sigma_k}\mathtt{f}^{(1)}(y)-T_{\sigma_k}\mathtt{f}^{(1)}(x) \big|^r\;dy\bigg)^{\frac{1}{r}}\right].
 \end{align*}
It follows from Proposition \ref{linearkeypropo1} (2) that
\begin{equation*}
 \bigg(\frac{1}{|Q|}\int_Q \big|T_{\sigma_k}\mathtt{f}^{(0)}(y)\big|^r\; dy\bigg)^{\frac{1}{r}}  \lesssim \big(2^k\ell(Q)\big)^{\lambda n \frac{1-\rho}{r(1-\lambda)} }\mathrm{M}_rf(x).
\end{equation*}
In addition, using \eqref{linearkeypropo2_2},  for $y\in Q$
 \begin{equation*}
 \big|T_{\sigma_k}\mathtt{f}^{(1)}(y)-T_{\sigma_k}\mathtt{f}^{(1)}(x) \big|\lesssim \big(2^k\ell(Q) \big)^{1-\rho(N-\frac{n}{r})}\mathrm{M}_rf(x)
 \end{equation*}
 where we choose $\frac{n}{r}<N<\frac{n}{r}+\frac{1}{\rho}$. This proves \eqref{lineariiestm2} for $\mathrm{i}=1$.\\

 For the term $\mathcal{J}_2$, assume that
 \begin{equation*}%\label{musel}
 1\le 2^k\ell(Q) \q \text{ and }\q 2^{\rho k}\ell(Q)<1.
 \end{equation*}
  We choose a positive number $\epsilon$ such that
 \begin{equation}\label{linrangeep}
 \lambda \Big(\frac{1-\rho}{1-\lambda}\Big)<\epsilon<\rho,
 \end{equation}
 which is possible because $\lambda<\rho$.
 Now let $Q_{\epsilon,k}^{\rho}$ be the concentric dilate of $Q$ with $\ell(Q_{\epsilon,k}^{\rho})=10\sqrt{n} \ell(Q)^{\rho} \big(2^{k}\ell(Q) \big)^{-\epsilon}$. Here, we point out that
 \begin{equation}\label{10sqnlqlelqrho}
 10 \sqrt{n}  \ell(Q)\le \ell(Q_{\epsilon,k}^{\rho}).
 \end{equation}
 Indeed, it is equivalent to
 \begin{align*}
 2^{\rho k}\ell(Q)^{\rho(\frac{1-\rho+\epsilon}{\epsilon})}\le 1
 \end{align*}
 and this is true as 
 $$\rho\Big(\frac{1-\rho+\epsilon}{\epsilon}\Big)\ge 1,$$
 equivalently,
 $$\epsilon\le \rho.$$
Then we divide
 $$f=f\chi_{Q_{\epsilon,k}^{\rho}}+f\chi_{(Q_{\epsilon,k}^{\rho})^c}=:f_{\rho,k}^{(0)}+f_{\rho,k}^{(1)}$$
  so that $T_{\sigma_k}f$ can be written as
  \begin{equation*}
  T_{\sigma_k}f=T_{\sigma_k}f_{\rho,k}^{(0)}+ T_{\sigma_k} f_{\rho,k}^{(1)}.
  \end{equation*}
Now by using Proposition \ref{linearkeypropo1} (2), 
\begin{align*}
\bigg( \frac{1}{|Q|}\int_Q \big| T_{\sigma_k}f_{\rho,k}^{(0)}(y)\big|^r\; dy \bigg)^{\frac{1}{r}}&\lesssim \big( 2^{ k}\ell(Q)\big)^{ \lambda n \frac{1-\rho}{r(1-\lambda)}    }\big(2^k\ell(Q) \big)^{-\frac{\epsilon n }{r}}\mathrm{M}_rf(x)\\
&=\big(2^k\ell(Q) \big)^{-\frac{n}{r}(\epsilon-\lambda(\frac{1-\rho}{1-\lambda}))}\mathrm{M}_rf(x).
\end{align*}
Furthermore, using \eqref{linearkeypropo2_1},
\begin{align*}
\big| T_{\sigma_k}f_{\rho,k}^{(1)}(y)\big|& \lesssim \Big( 2^{k\rho}\ell(Q)^{\rho}\big( 2^k\ell(Q)\big)^{-\epsilon}    \Big)^{-(N-\frac{n}{r})}\mathrm{M}_rf(x)\\
&= \big(2^k\ell(Q) \big)^{-(\rho-\epsilon)(N-\frac{n}{r})} \mathrm{M}_rf(x)
\end{align*} where $N>\frac{n}{r}$.
By \eqref{linrangeep}, these clearly imply \eqref{lineariiestm2} for $\ii=2$.\\

Now let us consider the last one $\mathcal{J}_3$. In this case, we will use the decomposition \eqref{lineardildecfg} to split $T_{\sigma_k}f$ as
 \begin{equation*}
 T_{\sigma_k}f=T_{\sigma_k}\mathbf{f}^{(0)}+T_{\sigma_k}\mathbf{f}^{(1)}.
  \end{equation*}
 Since
 $$2^{\frac{kn}{2}(1-\rho)}\sigma_k\in S_{\rho,\rho}^{-n(1-\rho)(\frac{1}{r}-\frac{1}{2})}(\bbrn) ~ \text{uniformly in } k,$$
 the $L^r$ boundedness for $T_{\sigma_k}$ in Theorem \ref{thma} implies
\begin{align*}
  \bigg(\frac{1}{|Q|}\int_Q \big| T_{\sigma_k}\mathbf{f}^{(0)}(y)\big|^r \; dy\bigg)^{\frac{1}{r}}&\le \frac{1}{|Q|^{\frac{1}{r}}} \big\Vert T_{\sigma_k}\mathbf{f}^{(0)}\big\Vert_{L^{r}(\bbrn)}\\
 &\lesssim 2^{-\frac{kn}{2}(1-\rho)} \frac{1}{\ell(Q)^{\frac{n}{r}}}\Vert f\Vert_{L^r(Q^*)}\lesssim 2^{-\frac{kn}{2}(1-\rho)}\mathrm{M}_rf(x).
 \end{align*}
 This proves
 \begin{equation*}
 \sum_{k:2^{\rho k}\ell(Q)\ge 1}  \bigg( \frac{1}{|Q|}\int_Q \big|   T_{\sigma_k}\mathbf{f}^{(0)}(y)\big|^r\; dy\bigg)^{\frac{1}{r}} \le \sum_{k=0}^{\infty} \bigg( \frac{1}{|Q|}\int_Q \big|   T_{\sigma_k}\mathbf{f}^{(0)}(y)\big|^r\; dy\bigg)^{\frac{1}{r}} \lesssim \mathrm{M}_rf(x).
 \end{equation*}
Furthermore, we utilize \eqref{linearkeypropo2_1} to obtain the estimates that 
\begin{equation*}%\label{fk1est234}
\big| T_{\sigma_k}\mathbf{f}^{(1)}(y)\big| \lesssim \big( 2^{k\rho} \ell(Q)        \big)^{-(N-\frac{n}{r})}\mathrm{M}_rf(x)
\end{equation*}
where $N>\frac{n}{r}$. 
Hence,
\begin{equation*}
\sum_{k:2^{\rho k}\ell(Q)\ge 1}  \bigg( \frac{1}{|Q|}\int_Q \big| T_{\sigma_k}\mathbf{f}^{(1)}(y)\big|^r \; dy\bigg)^{\frac{1}{r}} \lesssim  \mathrm{M}_rf(x) \sum_{k:2^{\rho k}\ell(Q)\ge 1}\big( 2^{k\rho} \ell(Q)        \big)^{-(N-\frac{n}{r})}\lesssim \mathrm{M}_rf(x).
\end{equation*}
We finally arrive at the conclusion
 $$\mathcal{J}_3\lesssim   \mathrm{M}_rf(x)$$
and this completes the proof of \eqref{linearmest1}.

 \hfill
 
 \section{Decomposition of multilinear pseudo-differential operators}\label{decompsection}

We will basically treat the multilinear operators in a structure similar to the linear case in Section \ref{linearproof}, using Littlewood-Paley decomposition in frequency spaces.
 As a multilinear counterpart of the partition of unity, we consider
  Schwartz functions $\Phi$ and $\Psi$ on $(\bbrn)^l$ such that
  $$\supp(\wh{\Phi})\subset \big\{\xxxi\in (\bbrn)^l: |\xxxi|\le 2 \big\}, \quad \supp(\wh{\Psi})\subset \big\{\xxxi\in (\bbrn)^l: 1/2\le |\xxxi|\le 2 \big\},$$
 and
  \begin{equation*}%\label{LPdecom}
 \wh{\Phi}(\xxxi)+\sum_{k\in\bbn}\wh{\Psi_k}(\xxxi)=1
\end{equation*}
where $\Psi_k(\xxx):=2^{knl}\Psi(2^k\xxx)$ for each $k\in\bbn$.
 Then the symbol $\sigma \in \mathbb{M}_l S^m_{\rho,\rho}(\bbrn)$ can be decomposed as
$$\sigma(x,\xxxi)=\sigma(x,\xxxi)\wh{\Phi}(\xxxi)+\sum_{k\in\bbn}\sigma(x,\xxxi)\wh{\Psi_k}(\xxxi)=:\sigma_0(x,\xxxi)+\sum_{k\in\bbn}\sigma_k(x,\xxxi)$$ for each $x\in\bbrn$ and $\xxxi\in (\bbrn)^l$
and accordingly we have
\begin{equation}\label{sigmadecomp}
T_{\sigma}(f_1,\dots,f_l)=\sum_{k\in\bbn_0}T_{\sigma_k}(f_1,\dots,f_l)
\end{equation}
 where $T_{\sigma_k}$ are the multilinear pseudo-differential operators associated with $\sigma_k$. We need to notice that for each multi-indices $\alpha,\beta_1,\dots,\beta_l\in (\bbn_0)^n$
  there exists a constant $C_{\alpha,\beta_1,\dots,\beta_l}>0$, independent of $k$, such that
\begin{equation}\label{sigmakptest}
\big|\partial_x^{\alpha}\partial_{\xi_1}^{\beta_1}\cdots\partial_{\xi_l}^{\beta_l}\sigma_k(x,\xi_1,\dots,\xi_l)\big|\le C_{\alpha,\beta_1,\dots,\beta_l}\big( 1+|\xi_1|+\cdots+|\xi_l|\big)^{m-\rho(|\beta_1|+\dots+|\beta_l|-|\alpha|)}.
\end{equation}
 In other words, 
  $\sigma_k\in \mathbb{M}_lS_{\rho,\rho}^{m}(\bbrn)$ uniformly in $k$.
  Now we write
$$T_{\sigma_k}\big(f_1,\dots,f_l \big)(y)=\int_{(\bbrn)^l}K_k(y,y-u_1,\dots,y-u_l) \prod_{j=1}^{l}f_j(u_j) \; du_1\cdots du_l$$
 where
 \begin{equation}\label{kkyuvdef}
 K_k(y,u_1,\dots,u_l):=\int_{(\bbrn)^l}\sigma_k(y,\xi_1,\dots,\xi_l)e^{2\pi i\sum_{j=1}^{l}\langle u_j,\xi_j\rangle}\; d\xi_1\cdots d\xi_l.
 \end{equation}

For any $0\le \lambda\le \rho$, we define
\begin{equation}\label{taudef}
\tau_k(x,\xi_1,\dots,\xi_l):=\sigma_k(2^{-\lambda k}x, 2^{\lambda k}\xi_1,\dots, 2^{\lambda k}\xi_l).
\end{equation}
 Then in view of \eqref{sigmakptest}, for any multi-indices $\alpha,\beta_1,\dots,\beta_l\in (\bbn_0)^n$ there is a constant $C_{\alpha,\beta_1,\dots,\beta_l}>0$, independent of $k$, such that
  \begin{equation}\label{ckins}
 \big| \partial_x^{\alpha}\partial_{\xi_1}^{\beta_1}\cdots\partial_{\xi_l}^{\beta_l}\tau_k(x,\xi_1,\dots,\xi_l)     \big|\le C_{\alpha,\beta_1,\dots,\beta_l} \big( 1+|\xi_1|+\dots+|\xi_l|\big)^{\frac{m}{1-\lambda}-\frac{\rho-\lambda}{1-\lambda}(|\beta_1|+\dots+|\beta_l|-|\alpha|)}.
 \end{equation} 
 That is,
 \begin{equation}\label{taukinmlsrn}
 \tau_k\in \mathbb{M}_lS_{\frac{\rho-\lambda}{1-\lambda},\frac{\rho-\lambda}{1-\lambda}}^{\frac{m}{1-\lambda}}(\bbrn)\q \text{ uniformly in }~k.
 \end{equation} 
 The proof of \eqref{ckins} is straightforward and almost same as that of \eqref{linearckins}, which is provided in the appendix, so it is omitted here. 
 Moreover, it is easy to verify that
 \begin{equation}\label{tsigkf1lest}
 T_{\sigma_k}\big(f_1,\dots,f_l \big)(x)=T_{\tau_k}\big( f_1(2^{-\lambda k}\cdot),\dots,f_l(2^{-\lambda k}\cdot)\big)(2^{\lambda k}x).
 \end{equation}

 As a multilinear counterpart of Lemma \ref{linearkeyest}, we have the following kernel estimates, which will play an essential role in establishing pointwise estimates in Proposition \ref{keypropo2} in Section \ref{mefmosec}.
\begin{lemma}\label{kernelest}
Let $0\le \rho<1$ and $m\in\bbr$. 
Suppose that $\sigma\in \mathbb{M}_lS_{\rho,\rho}^{m}(\bbrn)$ and let $K_k$ be defined as in \eqref{kkyuvdef}.
For arbitrary nonnegative $N_1,\dots,N_l\ge 0$ and $1\le r\le 2$, we have
 \begin{align*}
 \bigg\Vert \Big(\prod_{j=1}^{l}(1+2^{k\rho} |u_j|)^{N_j} \Big) K_k(y,u_1,\dots,u_l)\bigg\Vert_{L^{r'}({u_1,\dots,u_l \in \bbrn})}&\lesssim_{N_1,\dots,N_l}2^{k(m+\frac{nl}{r})}\\
  \bigg\Vert \Big( \prod_{j=1}^{l}(1+2^{k\rho} |u_j|)^{N_j}\Big)\nabla_y K_k(y,u_1,\dots,u_l) \bigg\Vert_{L^{r'}(u_1,\dots,u_l \in \bbrn)}&\lesssim_{N_1,\dots,N_l}2^{k(\rho+m+\frac{nl}{r})}\\
    \bigg\Vert \Big( \prod_{j=1}^{l}(1+2^{k\rho} |u_j|)^{N_j}\Big)\nabla_{u_{\mathrm{i}}} K_k(y,u_1,\dots,u_l) \bigg\Vert_{ L^{r'}(u_1,\dots,u_l \in \bbrn)}&\lesssim_{N_1,\dots,N_l}2^{k(1+m+\frac{nl}{r})},
\end{align*}
 uniformly in $y\in\bbrn$,
  where $\nabla_y$ and $\nabla_{u_{\mathrm{i}}}$,  $\mathrm{i}=1,\dots,l$, denote the gradient operators with respect to $y$ and $u_{\mathrm{i}}$, respectively.
\end{lemma}
The proof of the lemma is exactly same as that of Lemma \ref{linearkeyest}, simply using the multilinear estimates
   \begin{align*}
 \big|\partial_{\xi_1}^{\beta_1}\cdots \partial_{\xi_l}^{\beta_l}\sigma_k(y,\xi_1,\dots,\xi_l) \big|&\lesssim 2^{k(m-\rho(|\beta_1|+\dots+|\beta_l|))}\chi_{\{ 1+|\xi_1|+\cdots+|\xi_l|\sim 2^k\}}\\
  \big|\partial_{\xi_1}^{\beta_1}\cdots \partial_{\xi_l}^{\beta_l}\nabla_y\sigma_k(y,\xi_1,\dots,\xi_l) \big|&\lesssim 2^{k(m+\rho-\rho(|\beta_1|+\dots+|\beta_l|))}\chi_{\{1+|\xi_1|+\dots+|\beta_l|\sim 2^k\}}\\
    \big|\partial_{\xi_1}^{\beta_1}\cdots\partial_{\xi_l}^{\beta_l} \big( \xi_{\mathrm{i}} \,\sigma_k(y,\xi_1,\dots,\xi_l)\big) \big|&\lesssim 2^{k(m+1-\rho(|\beta_1|+\dots+|\beta_l|))}\chi_{\{1+|\xi_1|+\dots+|\xi_l|\sim 2^k\}}
 \end{align*}  for $\mathrm{i}=1,\dots,l$, in place of \eqref{paxibesikest1} and \eqref{paxibesikest2}.

\hfill

 \section{Proof of Proposition \ref{multilineargerho}}\label{thm12pf}
 
Let $0<\rho<1$ and $0<p_1,\dots,p_l\le \infty$ with $1/p_1+\dots+1/p_l=1/p$. Suppose that
 $m<m_{\rho}(\ppp)$ and $\sigma\in \mathbb{M}_lS_{\rho,\rho}^{m}(\bbrn)$
 where $\ppp:=(p_1,\dots,p_l)$.
  Let $\epsilon_0:=m_{\rho}(\ppp)-m>0$.
 Setting $\lambda=\rho$ and using \eqref{ckins}, we first see
  $$2^{k\epsilon_0}\tau_k\in \mathbb{M}_lS_{0,0}^{m_0(\ppp)}(\bbrn) \text{ uniformly in }~k.$$
 Then it follows from Theorem \ref{thme} that
 $$ 2^{k\epsilon_0}\big\Vert T_{\tau_k}(f_1,\dots,f_l)\big\Vert_{Y^p(\bbrn)}\lesssim \prod_{j=1}^{l}\Vert f_j\Vert_{Y^{p_j}(\bbrn)} \q \text{uniformly in }~k$$
 for all Schwartz functions $f_1,\dots,f_l$ on $\bbrn$, and this further implies that
 $$\big\Vert T_{\tau_k}(f_1,\dots,f_l)\big\Vert_{X^p(\bbrn)}\lesssim 2^{-k\epsilon_0} \prod_{j=1}^{l}\Vert f_j\Vert_{H^{p_j}(\bbrn)},$$
 due to the estimate
 $$\big\Vert T_{\tau_k}(f_1,\dots,f_l)\big\Vert_{X^p(\bbrn)}\lesssim \big\Vert T_{\tau_k}(f_1,\dots,f_l)\big\Vert_{Y^p(\bbrn)}$$
 and the embedding $H^{q}\hookrightarrow Y^q$ for $0<q\le \infty$,
 where $X^p$ is defined as in \eqref{xpspacedef}.
 Then by using \eqref{tsigkf1lest} and dilation properties of norms in $X^q$ and $H^q$, namely
 $$\big\Vert f(A\cdot)\big\Vert_{X^q(\bbrn)}=A^{-\frac{n}{q}}\Vert f\Vert_{X^q(\bbrn)}\q \text{and}\q \big\Vert f(A\cdot)\big\Vert_{H^q(\bbrn)}=A^{-\frac{n}{q}}\Vert f\Vert_{H^q(\bbrn)} \q \text{for}~A>0,$$
 we obtain
 \begin{align*}
 \big\Vert T_{\sigma_k}(f_1,\dots,f_l)\big\Vert_{X^p(\bbrn)}&=\big\Vert T_{\tau_k}\big( f_1(2^{-\rho k}\cdot),\dots,f_l(2^{-\rho k}\cdot)\big)(2^{\rho k}\cdot)\big\Vert_{X^p(\bbrn)}\\
 &=2^{-k\frac{\rho n}{p}}\big\Vert T_{\tau_k}\big( f_1(2^{-\rho k}\cdot),\dots,f_l(2^{-\rho k}\cdot)\big)\big\Vert_{X^p(\bbrn)}\\
 &\lesssim 2^{-k\epsilon_0}2^{-k\frac{\rho n}{p}}\prod_{j=1}^{l}\big\Vert f_j(2^{-\rho k}\cdot)\big\Vert_{H^{p_j}(\bbrn)}\\
 &=2^{-k\epsilon_0} \prod_{j=1}^{l}\Vert f_j\Vert_{H^{p_j}(\bbrn)}.
 \end{align*}
 Finally, we have
 \begin{align*}
 \big\Vert T_{\sigma}(f_1,\dots,f_l)\big\Vert_{X^p(\bbrn)}&\le \Big( \sum_{k\in\bbn_0} \big\Vert T_{\sigma_k}(f_1,\dots,f_l)\big\Vert_{X^p(\bbrn)}^{\min(1,p)}   \Big)^{1/\min(1,p)}\\
 &\lesssim \Big( \sum_{k\in\bbn_0}{ 2^{-k\epsilon_0 \min(1,p)}} \Big)^{1/\min(1,p)}\prod_{j=1}^{l}\Vert f_j\Vert_{H^{p_j}(\bbrn)}\\
 &\lesssim \prod_{j=1}^{l}\Vert f_j\Vert_{H^{p_j}(\bbrn)}
 \end{align*}
 where we used the convexity of $\Vert \cdot\Vert_{X^p}$ when $p\ge 1$ and of $\Vert \cdot \Vert_{X^p}^p$ when $p<1$.
 This completes the proof.

 \section{Main estimates for multilinear operators }\label{mefmosec}
 
 In this section, we introduce a technically convenient tool which will structurally convert  multilinear problems into linear forms, and study several estimates
 that will be central to the proof of Theorem \ref{mainpointest}.

\subsection{Generalized trace theorem for Sobolev spaces}

For $s>0$ and $N\in\bbn$, let $(I_N-\Delta_N)^{s/2}$ be the Bessel potential acting on  tempered distributions on $\bbr^N$, defined by
$$(I_N-\Delta_N)^{s/2}f:= \big( (1+4\pi^2  |\cdot|^2)^{s/2}\wh{f}\big)^{\vee}.$$
The Sobolev space $L^2_s(\mathbb{R}^N)$ consists of all $f \in L^2(\mathbb{R}^N)$ satisfying
$$\Vert f\Vert_{L^2_s(\bbr^N)}:=\big\Vert (I_N-\Delta_N)^{s/2}f\big\Vert_{L^2(\bbr^N)}=\Big(\int_{\bbr^N}\big(1+ 4\pi^2|\xi|^2 \big)^s\big|\wh{f}(\xi) \big|^2\; d\xi \Big)^{1/2}<\infty.$$

 \begin{lemma}\label{tracethm}
Let $l\ge 2$ and $s>0$.
 Then
 \begin{equation}\label{traceest}
 \Vert \widetilde{G} \Vert_{L^2_{s}(\bbrn)}
\lesssim_{s,r,n} \Vert G\Vert_{L^2_{s+{(l-1)n}/{2}}((\bbrn)^l)}
 \end{equation}
 for Schwartz functions $G$ on $(\bbrn)^l$,
 where $\widetilde{G}(x)=G(x_1,\dots,x_{l}) |_{x_1=\dots=x_l=x}$ and $x,x_1,\dots,x_l \in \bbrn$.
 \end{lemma}
 \begin{proof}
 
 We first claim that for each $k=1,\dots,l-1$ and $s>0$,
 \begin{equation}\label{reductionclaim}
 \big\Vert H(\cdot_1,\dots,\cdot_k,\cdot_k)\big\Vert_{L^2_s((\bbrn)^k)}\lesssim \big\Vert H(\cdot_1,\dots,\cdot_k,\cdot_{k+1})\big\Vert_{L^2_{s+n/2}((\bbrn)^{k+1})}
 \end{equation}
 for any Schwartz functions $H$ on $(\bbrn)^{k+1}$.
 For convenience, we will use the notations
 \begin{align*}
 \llbracket \xi_1,\dots,\xi_k \rrbracket&:=\Big( 1+4\pi^2\big( |\xi_1|^2+\dots+|\xi_k|^2\big)\Big)^{1/2},\\
 \llbracket \xi_1,\dots,\xi_k,\xi_{k+1}\rrbracket&:=\Big(1+4\pi^2\big( |\xi_1|^2+\dots+ |\xi_k|^2+ |\xi_{k+1}|^2\big)\Big)^{1/2},
 \end{align*}
 and $\mathcal{F}_{x_{k+1}\to \xi_{k+1}}$ for the Fourier transform of a function with respect to the variable $x_{k+1}$ in $\xi_{k+1}$ variable. 
 Similarly, the symbol $\mathcal{F}_{\xxx\to \xxxi}$ denotes  the Fourier transform of a function with respect to the variable $\xxx:=(x_1,\dots,x_k)$ in $\xxxi:=(\xi_1,\dots,\xi_k)$ variable.
Since
$$H(x_1,\dots,x_k,x_k)=\int_{\bbrn}  \mathcal{F}_{x_{k+1}\to\xi_{k+1}}\big(H(x_1,\dots,x_k,x_{k+1}) \big)(\xi_{k+1})  e^{2\pi i\langle x_k,\xi_{k+1}\rangle}   \; d\xi_{k+1},$$
we have
\begin{align*}
&\mathcal{F}_{\xxx\to\xxxi}\big(H(x_1,\dots,x_k,x_k)\big)(\xi_1,\dots,\xi_k)=\int_{(\bbrn)^k}H(x_1,\dots,x_k,x_k)e^{-2\pi i\langle \xxx,\xxxi\rangle}\; d\xxx\\
&=\int_{\bbrn} \int_{(\bbrn)^k}\mathcal{F}_{x_{k+1}\to\xi_{k+1}}\big(H(x_1,\dots,x_k,x_{k+1})\big)(\xi_{k+1})e^{-2\pi i [(\sum_{j=1}^{k-1}\langle x_j,\xi_j \rangle)+ \langle x_k,\xi_k-\xi_{k+1}\rangle   ]}\; d\xxx d\xi_{k+1}\\
&=\int_{\bbrn}   \wh{H}(\xi_1,\dots,\xi_{k-1},\xi_k-\xi_{k+1},\xi_{k+1})   \; d\xi_{k+1}.
\end{align*}
 By using the Cauchy-Schwarz inequality,
\begin{align*}
&\big| \mathcal{F}_{\xxx\to\xxxi}\big(H(x_1,\dots,x_k,x_k)\big)(\xi_1,\dots,\xi_k)\big| \\
 &\le \Big(\int_{\bbrn} \llbracket \xi_1,\dots,\xi_{k-1},\xi_k-\xi_{k+1},\xi_{k+1}\rrbracket^{-2s-n}\; d\xi_{k+1} \Big)^{1/2}\\
 &\q\times \Big(\int_{\bbrn} \llbracket \xi_1,\dots,\xi_{k-1},\xi_k-\xi_{k+1},\xi_{k+1}\rrbracket^{2s+n} \big|\wh{H}(\xi_1,\dots,\xi_{k-1},\xi_k-\xi_{k+1},\xi_{k+1})\big|^2  \; d\xi_{k+1} \Big)^{1/2}.
\end{align*}
Since 
$$ \llbracket \xi_1,\dots,\xi_{k-1},\xi_k-\xi_{k+1},\xi_{k+1}\rrbracket \sim \llbracket \xi_1,\dots, \xi_{k-1}, \xi_k,\xi_{k+1}\rrbracket,$$
we have 
 \begin{align*}
 &\Big(\int_{\bbrn} \llbracket \xi_1,\dots,\xi_{k-1},\xi_k-\xi_{k+1},\xi_{k+1}\rrbracket^{-2s-n}\; d\xi_{k+1} \Big)^{1/2}\\
 &\sim \Big( \int_{\bbrn}  \llbracket \xi_1,\dots,\xi_{k+1}\rrbracket^{-2s-n} \; d\xi_{k+1} \Big)^{1/2}\lesssim \llbracket \xi_1,\dots,\xi_k \rrbracket^{-s}
  \end{align*}
  where the inequality follows from the fact that
  $$\int_{\bbrn} {(A+|x|)^{-n-t}}\; dx\lesssim A^{-t}$$
  for $t>0$ and $A\ge 1$.
Hence, 
\begin{align*}
&\llbracket \xi_1,\dots,\xi_k \rrbracket^{s}\big| \mathcal{F}_{\xxx\to\xxxi}\big(H(x_1,\dots,x_k,x_k)\big)(\xi_1,\dots,\xi_k)\big| \\
&\lesssim  \Big(\int_{\bbrn} \llbracket \xi_1,\dots,\xi_{k-1},\xi_k-\xi_{k+1},\xi_{k+1}\rrbracket^{2s+n} \big|\wh{H}(\xi_1,\dots,\xi_{k-1},\xi_k-\xi_{k+1},\xi_{k+1})\big|^2  \; d\xi_{k+1} \Big)^{1/2}
\end{align*}
  and this finally yields that
  \begin{align*}
  \big\Vert H(\cdot_1,\dots,\cdot_k,\cdot_k)\big\Vert_{L^2_s((\bbrn)^k)}&\lesssim \Big(\int_{(\bbrn)^{k+1}} \llbracket \xi_1,\dots,\xi_{k+1}\rrbracket^{2s+n} \big|\wh{H}(\xi_1,\dots,\xi_{k+1})\big|^2  \; d\xxxi d\xi_{k+1} \Big)^{1/2}\\
  &=\big\Vert H(\cdot_1,\dots,\cdot_{k+1})\big\Vert_{L^2_{s+{n}/{2}}( (\bbrn)^{k+1})},
  \end{align*}
  which proves \eqref{reductionclaim}.
  
  Now let us complete the proof of \eqref{traceest}.
  Let $G$ be a Schwartz function on $(\bbrn)^l$ and $s>0$.
 Then for each $k=1,\dots,l-1$, by setting
 $$H(x_1,\dots,x_k,x_{k+1})=G(x_1,\dots,x_k,x_{k+1},x_{k+1},\dots,x_{k+1})$$
 and using \eqref{reductionclaim},
 we have
\begin{align*}
\big\| G(\cdot_1,\dots,\cdot_k,\cdot_k,\dots,\cdot_k)\big\|_{L^2_{s+(k-1)n/2}((\bbrn)^k)}
&=\big\| H(\cdot_1,\dots,\cdot_k,\cdot_k) \big\|_{L^2_{s+(k-1)n/2}((\bbrn)^k)}
\\
&\lesssim
\big\| H(\cdot_1,\dots,\cdot_k,\cdot_{k+1}) \big\|_{L^2_{s+kn/2}((\bbrn)^{k+1})}
\\
&= \big\| G(\cdot_1,\dots,\cdot_k,\cdot_{k+1},\dots,\cdot_{k+1})\big\|_{L^2_{s+kn/2}((\bbrn)^{k+1})}.
 \end{align*}
Therefore, starting at $k=1$, we conclude that
\begin{align*}
 \big\Vert G(\cdot,\dots,\cdot)\big\Vert_{L^2_s(\bbrn)}&\lesssim \big\Vert G(\cdot_1,\cdot_2,\cdot_2,\dots,\cdot_2)\big\Vert_{L^2_{s+n/2}((\bbrn)^2)}\\
 &\vdots\\
 &\lesssim \big\Vert G(\cdot_1,\dots,\cdot_k,\cdot_k,\dots,\cdot_k)\big\Vert_{L^2_{s+(k-1)n/2}((\bbrn)^k)}\\
 &\vdots\\
 &\lesssim \big\Vert G(\cdot_1,\dots,\cdot_l)\big\Vert_{L^2_{s+{(l-1)n}/{2}}((\bbrn)^l)}.
\end{align*}
 This finishes the proof.
 \end{proof}

 \subsection{Key Estimates}
 
 Let $P$ be the concentric dilate of $Q$ with $\ell(P)\ge 10\sqrt{n} l \ell(Q)$ and $\chi_P$ indicate the characteristic function of $P$.
We divide 
 \begin{align*}
 f_j= f_j\chi_P+f_j\chi_{P^c}=:f_j^{(0)}+f_j^{(1)} \q \text{ for }~j=1,\dots,l
 \end{align*}
 so that $T_{\sigma_k}(f_1,\dots,f_l)$ can be expressed as the sum of $2^{l}$ different terms of the form
 $$T_{\sigma_k}\big(f_1^{(\mu_1)},f_2^{(\mu_2)},\dots,f_l^{(\mu_l)}\big)$$
 where each $\mu_j$ is $0$ or $1$.

 We now state multilinear analogues of Propositions \ref{linearkeypropo1} and \ref{linearkeypropo2}, which will play a key role in the proof of Theorem \ref{mainpointest}.
 \begin{proposition}\label{keypropo1}
Let $0< \rho<1$, $1< r\le 2$,  $m=-\frac{nl}{r}(1-\rho)$, and $k\in\bbn_0$. Suppose that $x\in Q$.
\begin{enumerate}
\item Suppose that $0< \rho<\frac{r}{2l}$.
Then every $\sigma\in \mathbb{M}_lS_{\rho,\rho}^{m}(\bbrn)$ satisfies
 $$\frac{1}{|Q|}\int_Q \big| T_{\sigma_k}\big(f_1^{(0)},\dots,f^{(0)}_l \big)(y)\big| \; dy\lesssim \Big(\frac{\ell(P)}{\ell(Q)^{\rho}} \Big)^{\frac{nl}{r}}\mathbf{M}_r\big(f_1,\dots,f_l\big)(x).$$

\item 
Suppose that $\frac{r}{2l}\le \rho <1$ and $\frac{2\rho l-r}{2l-r}<\lambda<\rho$.
 Then every $\sigma\in \mathbb{M}_lS_{\rho,\rho}^{m}(\bbrn)$ satisfies
 $$\frac{1}{|Q|}\int_Q \big| T_{\sigma_k}\big(f_1^{(0)},\dots,f_l^{(0)} \big)(y)\big| \; dy\lesssim \big( 2^k \ell(Q)\big)^{\lambda n \frac{l(1-\rho)}{r(1-\lambda)} }\Big(\frac{\ell(P)}{\ell(Q)^{\rho}} \Big)^{\frac{nl}{r}}\mathbf{M}_r\big(f_1,\dots,f_l\big)(x).$$
 \end{enumerate}
 \end{proposition}

 \begin{proposition}\label{keypropo2}
 Let $0\le \rho<1$, $1\le r\le 2$, $m=-\frac{nl}{r}(1-\rho)$, and $k\in\bbn_0$. Suppose that $x,y\in Q$.
 Then every $\sigma\in \mathbb{M}_lS_{\rho,\rho}^{m}(\bbrn)$ satisfies
 \begin{equation}\label{keypropo2_1}
 \big| T_{\sigma_k}\big(f_1^{(\mu_1)},\dots,f_l^{(\mu_l)}\big)(y)\big|\lesssim_N \big( 2^{k\rho} \ell(P)\big)^{-(N-\frac{nl}{r})}\mathbf{M}_r\big(f_1,\dots,f_l\big)(x)
\end{equation}
and
\begin{align}\label{keypropo2_2}
& \big| T_{\sigma_k}\big(f_1^{(\mu_1)},\dots,f_l^{(\mu_l)}\big)(y)-T_{\sigma_k}\big(f_1^{(\mu_1)},\dots,f_l^{(\mu_l)}\big)(x)\big|\nonumber\\
&\lesssim_N 2^k\ell(Q)\big( 2^{k\rho} \ell(P)\big)^{-(N-\frac{nl}{r})}\mathbf{M}_r\big(f_1,\dots,f_l\big)(x)
\end{align}
for any $N>\frac{nl}{r}$, when $\mu_j$ is $0$ or $1$, excluding $\mu_1=\cdots=\mu_l=0$.
\end{proposition}

 In order to prove the propositions, we first need the following lemma, which is a generalization of Lemmas \ref{MiYapropo} and \cite[Lemma 2.1]{Na2015} that dealt with the case $r=2$ and $l=2$.
    \begin{lemma}\label{Naibopropo}
  Let $1< r\le 2$, $0< \rho< \frac{r}{2l}$, and $m=-\frac{nl}{r}(1-\rho)$.
 Then every $\sigma\in \mathbb{M}_lS_{\rho,\rho}^m(\bbrn)$ satisfies
 $$\big\Vert T_{\sigma}(f_1,\dots,f_l) \big\Vert_{L^{\frac{r}{l\rho}}(\bbrn)}\lesssim \prod_{j=1}^{l}\Vert f_j\Vert_{L^r(\bbrn)}$$
 for $f_1,\dots,f_l\in \mathscr{S}(\bbrn)$.
 \end{lemma}
 \begin{proof}
 The proof is based on the ideas of \cite[Lemma 5.6]{Mi_Ya1987} and \cite[Lemma 2.1]{Na2015}, using the generalized trace theorem in Lemma \ref{tracethm}.
 We first define the linear symbol $\Sigma$ in $\bbr^{nl}\times \bbr^{nl}$ as
 $$\Sigma(X,\Xi):=\sigma(x_1,\xi_1,\dots,\xi_l),\q X:=(x_1,\dots,x_l)\in (\bbr^{n})^l, ~ \Xi:=(\xi_1,\dots,\xi_l)\in (\bbr^{n})^l.$$
 Then it is obvious that the symbol $\Sigma$ belongs to the linear H\"ormander class $S_{\rho,\rho}^m(\bbr^{nl})$.
 We notice that the operator
 $$  (I-\Delta)^{\frac{1}{2}nl(\frac{1}{2}-\frac{\rho}{r})}T_{\Sigma} (I-\Delta)^{\frac{1}{2}nl(\frac{1}{r}-\frac{1}{2})}$$
 is a pseudo-differential operator associated with a symbol in $S_{\rho,\rho}^{0}(\bbr^{nl})$, as 
 $$nl\Big(\frac{1}{2}-\frac{\rho}{r}\Big)+m+nl\Big(\frac{1}{r}-\frac{1}{2}\Big)=0,$$
 where $\Delta$ is the Laplacian operator on $\bbr^{nl}$.
Then its $L^2(\bbr^{nl})$ boundedness, in view of Theorem \ref{thma}, deduces that
 \begin{align*}
 \big\Vert T_{\Sigma}(G)\big\Vert_{L^2_{nl(\frac{1}{2}-\frac{\rho}{r})}(\bbr^{nl})}
 &=\Big\Vert \Big( (I-\Delta)^{\frac{1}{2}nl(\frac{1}{2}-\frac{\rho}{r})}T_{\Sigma}(I-\Delta)^{\frac{1}{2}nl(\frac{1}{r}-\frac{1}{2})}    \Big)  (I-\Delta)^{-\frac{1}{2}nl(\frac{1}{r}-\frac{1}{2})}G   \Big\Vert_{L^2(\bbr^{nl})}\\
 &\lesssim   \Big\Vert  (I-\Delta)^{-\frac{1}{2}nl(\frac{1}{r}-\frac{1}{2})}G\Big\Vert_{L^2(\bbr^{nl})}              =\big\Vert G\big\Vert_{L^2_{-nl(\frac{1}{r}-\frac{1}{2})}(\bbr^{nl})}\lesssim \Vert G\Vert_{L^r(\bbr^{nl})}
 \end{align*}
 for any Schwartz functions $G$ on $\bbr^{nl}$,
 where the last inequality simply follows from the Sobolev embedding theorem, but we note that $r>1$.
 Therefore, by the embedding $L_{n(\frac{1}{2}-\frac{l\rho}{r})}^2(\bbrn)\hookrightarrow L^{\frac{r}{l\rho}}(\bbrn)$ and Lemma \ref{tracethm} with $s=n(\frac{1}{2}-\frac{l\rho}{r})$,
 we obtain
 \begin{align*}
 \big\Vert T_{\sigma}(f_1,\dots,f_l)\big\Vert_{L^{\frac{r}{l\rho}}(\bbrn)}&\lesssim \big\Vert T_{\sigma}(f_1,\dots,f_l)\big\Vert_{L_{n(\frac{1}{2}-\frac{l\rho}{r})}^2(\bbrn)}\\
 &\lesssim  \big\Vert T_{\Sigma}(f_1\otimes\cdots\otimes f_l)\big\Vert_{L^2_{nl(\frac{1}{2}-\frac{\rho}{r})}(\bbr^{nl})}\\
 &\lesssim \Vert f_1\otimes\cdots\otimes f_l\Vert_{L^r(\bbr^{nl})}=\prod_{j=1}^{l}\Vert f_j\Vert_{L^r(\bbrn)}
 \end{align*}
 where the condition $\frac{1}{2}-\frac{l\rho}{r}>0$ (equivalently, $\rho<\frac{r}{2l}$) is required in the second inequality.
This completes the proof.
 \end{proof}
 
The following lemma covers the remaining case $\frac{r}{2l}\le \rho<1$ where Lemma \ref{Naibopropo} cannot be applied.
  \begin{lemma}\label{rhoge12}
 Let $1< r\le 2$, $\frac{r}{2l} \le \rho<1$, $m=-\frac{nl}{r}(1-\rho)$, and $k\in\bbn_0$.
 Suppose that 
 $$\frac{2\rho l-r}{2l-r}<\lambda<\rho.$$
 Then every $\sigma\in \mathbb{M}_lS_{\rho,\rho}^{m}(\bbrn)$ satisfies
 \begin{equation*}%\label{rhoge12est}
 \big\Vert T_{\sigma_k}(f_1,\dots,f_l) \big\Vert_{L^{\frac{r(1-\lambda)}{l(\rho-\lambda)}}(\bbrn)}\lesssim 2^{\lambda n k \frac{l(1-\rho)}{r(1-\lambda)} }\prod_{j=1}^{l}\Vert f_j\Vert_{L^r(\bbrn)}
 \end{equation*}
 for $f_1,\dots,f_l\in\mathscr{S}(\bbrn)$.
 \end{lemma}
 \begin{proof}
 Let $\frac{2\rho l-r}{2l-r}<\lambda<\rho$ and $\tau_k$ be defined as in \eqref{taudef} via dilation.
Then we recall in \eqref{taukinmlsrn} that
 \begin{equation*}
 \tau_k\in \mathbb{M}_lS_{\frac{\rho-\lambda}{1-\lambda},\frac{\rho-\lambda}{1-\lambda}}^{\frac{m}{1-\lambda}}(\bbrn)\q \text{ uniformly in }~k.
 \end{equation*} 
 and observe that
$$0<\frac{\rho-\lambda}{1-\lambda}<\frac{r}{2l} \q \text{ and }\q \frac{m}{1-\lambda}=-\frac{nl}{r}\Big(1-\frac{\rho-\lambda}{1-\lambda}\Big).$$
 In view of \eqref{tsigkf1lest}, we have
 \begin{equation*}
 \big\Vert T_{\sigma_k}(f_1,\dots,f_l)\big\Vert_{L^{\frac{r(1-\lambda)}{l(\rho-\lambda)}}(\bbrn)}=2^{-\lambda nk \frac{l(\rho-\lambda)}{r(1-\lambda)}}\big\Vert T_{\tau_k}\big(f_1(2^{-\lambda k}\cdot),\dots,f_l(2^{-\lambda k}\cdot) \big)\big\Vert_{L^{\frac{r(1-\lambda)}{l(\rho-\lambda)}}(\bbrn)}
 \end{equation*}
and then
Lemma \ref{Naibopropo} yields that
\begin{align*}
\big\Vert T_{\tau_k}\big(f_1(2^{-\lambda k}\cdot),\dots,f_l(2^{-\lambda k}\cdot) \big)\big\Vert_{L^{\frac{r(1-\lambda)}{l(\rho-\lambda)}}(\bbrn)}&\lesssim \prod_{j=1}^{l}\big\Vert f_j(2^{-\lambda k}\cdot)\big\Vert_{L^r(\bbrn)}=2^{\lambda n k\frac{l}{r}}\prod_{j=1}^{l}\Vert f_j\Vert_{L^r(\bbrn)}
\end{align*}
where the constant in the inequality is independent of $k$.
Combining these estimates, we finally obtain
\begin{equation*}
 \big\Vert T_{\sigma_k}(f_1,\dots,f_l)\big\Vert_{L^{\frac{r(1-\lambda)}{l(\rho-\lambda)}}(\bbrn)}\lesssim 2^{\lambda n k \frac{l(1-\rho)}{r(1-\lambda)} }\prod_{j=1}^{l}\Vert f_j\Vert_{L^r(\bbrn)}
 \end{equation*}
 as desired.
 \end{proof}

 Now we prove Propositions \ref{keypropo1} and \ref{keypropo2}. The proofs of the propositions are quite similar to those of Propositions \ref{linearkeypropo1} and \ref{linearkeypropo2} as they are multilinear extensions of the previous ones.
 \begin{proof}[Proof of Proposition \ref{keypropo1}]
We first assume $0<\rho<\frac{r}{2l}$.
By H\"older's inequality and Lemma \ref{Naibopropo},
 \begin{align*}
 \frac{1}{|Q|}\int_Q \big| T_{\sigma_k}\big( f_1^{(0)},\dots,f_l^{(0)}\big)(y)\big|\;dy&\le \frac{1}{|Q|^{\frac{l\rho}{r}}}\big\Vert T_{\sigma_k}(f_1^{(0)},\dots,f_l^{(0)})\big\Vert_{L^{\frac{r}{l\rho}}(\bbrn)}\\
 &\lesssim \frac{1}{\ell(Q)^{\frac{nl\rho  }{r}}}\prod_{j=1}^{l}\Vert f_j\Vert_{L^r(P)}
 \end{align*}
 and we see
 \begin{equation}\label{l2onpest}
\prod_{j=1}^{l}\Vert f_j\Vert_{L^r(P)}=\Big( \int_{P^l} \big| f_1(u_1)\cdots f_l(u_l)\big|^{r}\; du_1\cdots du_l\Big)^{1/r}\lesssim \ell(P)^{\frac{ nl}{r}}\mathbf{M}_r\big(f_1,\dots,f_l\big)(x)
 \end{equation} 
 as $x\in Q\subset P$, where $P^l:=P\times\cdots\times P\subset (\bbrn)^l$.
 Then the first assertion follows.
 
Now we assume $\frac{r}{2l}\le \rho<1$ for the remaining estimate.
 By applying successively H\"older's inequality, Lemma \ref{rhoge12}, and \eqref{l2onpest},
  \begin{align*}
 \frac{1}{|Q|}\int_Q \big| T_{\sigma_k}\big( f_1^{(0)},\dots,f_l^{(0)}\big)(y)\big|\;dy&\le \frac{1}{|Q|^{\frac{l(\rho-\lambda)}{r(1-\lambda)}}}\big\Vert T_{\sigma_k}(f_1^{(0)},\dots,f_l^{(0)})\big\Vert_{L^{\frac{r(1-\lambda)}{l(\rho-\lambda)}}(\bbrn)}\\
 &\lesssim \frac{1}{\ell(Q)^{n \frac{l(\rho-\lambda)}{r(1-\lambda)}}}2^{\lambda n k \frac{l(1-\rho)}{r(1-\lambda)} }\prod_{j=1}^{l}\Vert f_j\Vert_{L^r(P)}\\
 &\lesssim 2^{\lambda n k \frac{l(1-\rho)}{r(1-\lambda)} }\ell(Q)^{-n \frac{l(\rho-\lambda)}{r(1-\lambda)} }\ell(P)^{\frac{nl}{r}}\mathbf{M}_r\big(f_1,\dots,f_l \big)(x)\\
 &=\big( 2^k \ell(Q)\big)^{\lambda n \frac{l(1-\rho)}{r(1-\lambda)} }\Big(\frac{\ell(P)}{\ell(Q)^{\rho}} \Big)^{\frac{nl}{r}}\mathbf{M}_r\big(f_1,\dots,f_l\big)(x).
 \end{align*}
This finishes the proof.
 \end{proof}

  \begin{proof}[Proof of Proposition \ref{keypropo2}]
 
We first consider \eqref{keypropo2_1}. By symmetry, we are only concerned with the cases when $\mu_1=\dots=\mu_{\kappa}=0$, $\mu_{\kappa+1}=\dots=\mu_l=1$ for $\kappa=1,\dots,l-1$
 and when $\mu_1=\dots=\mu_l=1$.
 Let $N>\frac{nl}{r}$, $y\in Q$, and $1\le \kappa\le l-1$.
 Using H\"older's inequality,
 \begin{align*}
 &\big| T_{\sigma_k}\big(f_1^{(0)},\dots,f_{\kappa}^{(0)},f_{\kappa+1}^{(1)},\dots,f_l^{(1)} \big)(y) \big|\\
 &\le \int_{u_1,\dots,u_{\kappa}\in P,u_{\kappa+1},\dots,u_l\in P^c}\big|K_k(y,y-u_1,\dots,y-u_l) \big| \big| f_1(u_1)\cdots f_l(u_l)\big|\; du_1\cdots du_l\\
 &\le \bigg\Vert \ell(P)^{\frac{nl}{r}}\bigg( \frac{|u_{\kappa+1}|+\dots+|u_l|}{\ell(P)}\bigg)^N K_k(y,u_1,\dots,u_l)\bigg\Vert_{ L^{r'}(u_1,\dots,u_l \in \bbrn)}\\
 &\q \times \bigg\Vert  \ell(P)^{-\frac{nl}{r}}\bigg( \frac{|y-u_{\kappa+1}|+\dots+|y-u_l|}{\ell(P)}\bigg)^{-N} f_1(u_1)\cdots f_l(u_l) \bigg\Vert_{L^r(u_1,\dots,u_{\kappa}\in P,u_{\kappa+1},\dots,u_l\in P^c)}.
 \end{align*}
 The first term is equal to
 $$\ell(P)^{-(N-\frac{nl}{r})}\Big\Vert \big(|u_{\kappa+1}|+\dots+|u_l|\big)^NK_k(y,u_1,\dots,u_l) \Big\Vert_{ L^{r'}(u_1,\dots,u_l \in \bbrn)}$$
 and this is clearly less than
 $$\ell(P)^{-(N-\frac{nl}{r})}2^{-k\rho N} \bigg\Vert \Big(\prod_{j=\kappa+1}^{l}\big( 1+2^{k\rho}|u_{j}|\big)^N\Big) K_k(y,u_1,\dots,u_l) \bigg\Vert_{ L^{r'}(u_1,\dots,u_l \in \bbrn)}.$$
 Then Lemma \ref{kernelest} yields that the preceding expression is bounded by a constant times
 \begin{equation*}
 \ell(P)^{-(N-\frac{nl}{r})}2^{-k\rho N}2^{k(m+\frac{nl}{r})}=\big(2^{k\rho}\ell(P)\big)^{-(N-\frac{nl}{r})},
 \end{equation*}
which proves
\begin{equation}\label{lpvnkkest}
\bigg\Vert \ell(P)^{\frac{nl}{r}}\bigg( \frac{|u_{\kappa+1}|+\dots+|u_l|}{\ell(P)}\bigg)^N K_k(y,u_1,\dots,u_l)\bigg\Vert_{ L^{r'}(u_1,\dots,u_l \in \bbrn)}\lesssim \big(2^{k\rho}\ell(P)\big)^{-(N-\frac{nl}{r})}.
\end{equation}
For the remaining term, we note that for $x,y\in Q$, $u_1,\dots,u_{\kappa}\in P$, and $u_{\kappa+1},\dots,u_{l}\in P^c$,
\begin{align*}
|y-u_{\kappa+1}|+\dots+|y-u_{l}| &\ge |x-u_{\kappa+1}|+\dots+|x-u_l|-l|x-y|\\
&\gtrsim_{l,n}  |x-u_{\kappa+1}|+\dots+|x-u_l| \\
&\gtrsim_{l,n} \ell(P)+|x-u_1|+\dots+|x-u_l|
\end{align*}
 and thus for $N>\frac{nl}{r}$
 \begin{align}\label{lpyvnest}
 &\bigg\Vert  \ell(P)^{-\frac{nl}{r}}\bigg( \frac{|y-u_{\kappa+1}|+\dots+|y-u_l|}{\ell(P)}\bigg)^{-N} f_1(u_1)\cdots f_l(u_l) \bigg\Vert_{L^r(u_1,\dots,u_{\kappa}\in P,u_{\kappa+1},\dots,u_l\in P^c)}\nonumber\\
 &\lesssim \bigg( \int_{(\bbrn)^l}\frac{1}{\ell(P)^{nl}}\frac{1}{\big(1+\frac{|x-u_1|+\dots+|x-u_l|}{\ell(P)} \big)^{rN}} \big| f_1(u_1)\cdots f_l(u_l)\big|^r      \; du_1\cdots du_l\bigg)^{1/r}\nonumber\\
 &\lesssim \mathbf{M}_r\big(f_1,\dots,f_l\big)(x).
 \end{align}
As a result, we have
\begin{equation*}
 \big| T_{\sigma_k}\big(f_1^{(0)},\dots,f_{\kappa}^{(0)},f_{\kappa+1}^{(1)},\dots,f_l^{(1)} \big)(y) \big| \lesssim \big(2^{k\rho}\ell(P)\big)^{-(N-\frac{nl}{r})}\mathbf{M}_r\big(f_1,\dots,f_l\big)(x),
 \end{equation*}
as claimed.

Similarly, we can also prove that
 \begin{align*}%\label{kkyuvpcpcest}
 &\big| T_{\sigma_k}\big(f_1^{(1)},\dots,f_l^{(1)} \big)(y) \big|\nonumber\\
 &\le \int_{u_1,\dots,u_l\in P^c}\big|K_k(y,y-u_1,\dots,y-u_l) \big| \big| f_1(u_1)\cdots f_l(u_l)\big|\; du_1\cdots du_l\nonumber\\
 &\le \bigg\Vert \ell(P)^{\frac{nl}{r}}\bigg( \frac{|u_1|+\dots+|u_l|}{\ell(P)}\bigg)^{N} K_k(y,u_1,\dots,u_l)\bigg\Vert_{ L^{r'}(u_1,\dots,u_l \in \bbrn)}\nonumber\\
 &\q \times \bigg\Vert  \ell(P)^{-\frac{nl}{r}}\bigg( \frac{|y-u_1|+\dots+|y-u_l|}{\ell(P)}\bigg)^{-N}f_1(u_1)\cdots f_l(u_l) \bigg\Vert_{L^r(u_1,\dots, u_l\in P^c)}\nonumber\\
 &\lesssim \big( 2^{k\rho}\ell(P)\big)^{-(N-\frac{nl}{r})}\mathbf{M}_r\big(f_1,\dots,f_l\big)(x).
 \end{align*}
  This completes the proof of \eqref{keypropo2_1}.

  We next consider \eqref{keypropo2_2}.
  This will be proved in a way similar to \eqref{keypropo2_1}, using
  $$\Gamma_k(y,x,u_1,\dots,u_l):=K_k(y,y-u_1,\dots,y-u_l)-K_k(x,x-u_1,\dots,x-u_l),$$
  instead of  $K_k(y,y-u_1,\dots,y-u_l)$.
   Let $x,y\in Q$ and $N>\frac{nl}{r}$. 
 We will only consider the case when at least one of $\mu_j$ is $0$, as the case when $\mu_1=\dots=\mu_l=1$ can be verified in the exactly same manner.
 Permuting the variables, without loss of generality, we may assume
 $\mu_1=\dots=\mu_{\kappa}=0$, $\mu_{\kappa+1}=\dots=\mu_l=1$ for $1\le \kappa\le l-1$. 
 Using H\"older's inequality, we have
 \begin{align*}
 &\big| T_{\sigma_k}\big(f_1^{(0)},\dots,f_{\kappa}^{(0)},f_{\kappa+1}^{(1)},\dots,f_l^{(1)}\big)(y)-T_{\sigma_k}\big(f_1^{(0)},\dots,f_{\kappa}^{(0)},f_{\kappa+1}^{(1)},\dots,f_l^{(1)}\big)(x)\big|\\
 &\le \int_{u_1,\dots,u_{\kappa}\in P, u_{\kappa+1},\dots,u_l\in P^c}\big| \Gamma_k(y,x,u_1,\dots,u_l)\big|\big| f_1(u_1)\cdots f_l(u_l)\big|\; du_1\cdots du_l\\
 &\le \bigg\Vert  \ell(P)^{\frac{nl}{r}}\Big( \frac{|y-u_{\kappa+1}|+\dots+|y-u_{l}|}{\ell(P)}\Big)^N\Gamma_k(y,x,u_1,\dots,u_l)\bigg\Vert_{L^{r'}(u_1,\dots, u_{\kappa}\in\bbrn, u_{\kappa+1},\dots,u_l\in P^c)}\\
&\q \times \bigg\Vert \ell(P)^{-\frac{nl}{r}}\Big( \frac{|y-u_{\kappa+1}|+\dots+|y-u_l|}{\ell(P)}\Big)^{-N}f_1(u_1)\cdots f_l(u_l)\bigg\Vert_{L^r(u_1,\dots,u_{\kappa}\in P,u_{\kappa+1},\dots,u_l\in P^c)}.
 \end{align*}
  Analogous to \eqref{lpvnkkest}, we claim that for $x,y\in Q$
 \begin{align}\label{lpnlrest}
 &\bigg\Vert  \ell(P)^{\frac{nl}{r}}\Big( \frac{|y-u_{\kappa+1}|+\dots+|y-u_l|}{\ell(P)}\Big)^N \Gamma_k(y,x,u_1,\dots,u_l)\bigg\Vert_{L^{r'}(u_1,\dots,u_{\kappa}\in\bbrn, u_{\kappa+1},\dots,u_l\in P^c)}\nonumber\\
 &\lesssim 2^k\ell(Q) \big( 2^{k\rho}\ell(P)\big)^{-(N-\frac{nl}{r})}.
 \end{align}
 Indeed, the left-hand side of \eqref{lpnlrest} is bounded by a constant multiple of
 \begin{align*}
 & \bigg\Vert \ell(Q) \ell(P)^{\frac{nl}{r}} \Big( \frac{|y-u_{\kappa+1}|+\dots+|y-u_l|}{\ell(P)}\Big)^N  \\
 &\q\q\q\times \int_0^1 \big| \nabla K_k\big(y(t),y(t)-u_1,\dots,y(t)-u_l \big)\big|\; dt      \bigg\Vert_{L^{r'}(u_1,\dots,u_{\kappa}\in \bbrn,u_{\kappa+1},\dots,u_l\in P^c)}\nonumber\\
 &\lesssim \ell(Q)\ell(P)^{-(N-\frac{nl}{r})}\int_0^1 \Big\Vert   \big(|y(t)-u_{\kappa+1}|+\dots+|y(t)-u_l|   \big)^N   \\
 &\qq\qq\qq\qq\qq\qq\times \big|\nabla K_k\big(y(t),y(t)-u_1,\dots, y(t)-u_l\big) \big|      \Big\Vert_{ L^{r'}(u_1,\dots,u_l \in \bbrn)} \; dt\nonumber\\
 \end{align*}
 where $y(t):=ty+(1-t)x\in Q$ so that $|y-v|\lesssim |y(t)-v|$ for $v\in P^c$.
 By performing a change of variables and applying  Lemma \ref{kernelest}, the $L^{r'}$-norm in the preceding expression is smaller than
 \begin{align*}
 &2^{-k\rho N} \bigg\Vert  \Big(\prod_{j=\kappa+1}^{l} \big(1+2^{k\rho}|u_{j}|\big)^N \Big)\Big|\nabla K_k\big(y(t),u_1,\dots, u_l\big) \Big|      \bigg\Vert_{ L^{r'}(u_1,\dots,u_l \in \bbrn)} \\
 &\lesssim   2^{-k\rho N}2^{k(1+m+\frac{nl}{r})}      =2^k 2^{-k\rho(N-\frac{nl}{r})}.
 \end{align*} 
 Combining all together, we obtain \eqref{lpnlrest}.
 Since \eqref{lpyvnest} can be used for the remaining term without any modification, we finally obtain
\begin{align*}
&\big| T_{\sigma_k}\big(f_1^{(0)},\dots,f_{\kappa}^{(0)},f_{\kappa+1}^{(1)},\dots,f_l^{(1)}\big)(y)-T_{\sigma_k}\big(f_1^{(0)},\dots,f_{\kappa}^{(0)},f_{\kappa+1}^{(1)},\dots,f_l^{(1)}\big)(x)\big|\\
&\lesssim  2^k\ell(Q) \big( 2^{k\rho}\ell(P)\big)^{-(N-\frac{nl}{r})}\mathbf{M}_r\big(f_1,\dots,f_l\big)(x).
\end{align*}
This completes the proof.
 \end{proof}

\hfill

 \section{Proof of Theorem \ref{mainpointest}}\label{thm14pf}

Let $1< r\le 2 \le l<\infty$. We fix a cube $Q$ and a point $x\in Q$. Then matters reduce to proving
\begin{equation}\label{mainestimate0}
\Big(\frac{1}{|Q|}\int_Q  \big|T_{\sigma}\big(f_1,\dots,f_l\big)(y)\big|^{\frac{r}{l}} \; dy \Big)^{\frac{l}{r}}\lesssim \mathbf{M}_r\big(f_1,\dots,f_l \big)(x) \q \text{ if }~\ell(Q)\ge 1
\end{equation}
and
\begin{equation}\label{mainestimate}
\inf_{c_Q\in\mathbb{C}}\Big(\frac{1}{|Q|}\int_Q  \big|T_{\sigma}\big(f_1,\dots,f_l\big)(y)-c_Q\big|^{\frac{r}{l}} \; dy \Big)^{\frac{l}{r}}\lesssim \mathbf{M}_r\big(f_1,\dots,f_l \big)(x) \q \text{ if }~ \ell(Q)<1
\end{equation}
 uniformly in $Q$ and $x$.

\subsection{Proof of \eqref{mainestimate0}} 
 Assume $\ell(Q)\ge 1$.
 We decompose
  \begin{equation}\label{dildecfg}
 f_j=f_j\chi_{Q^{**}}+f_j\chi_{(Q^{**})^c}=:\mathbf{f}_j^{(0)}+\mathbf{f}_j^{(1)} \q \text{ for }~j=1,\dots,l
 \end{equation}
 where $Q^{**}$ denotes the dilate of $Q$ by a factor of $10\sqrt{n} l$.
Then $T_{\sigma}(f_1,\dots,f_l)$ can be written as
 \begin{align*}
 T_{\sigma}(f_1,\dots,f_l)=\sum_{\mu_1,\dots,\mu_l\in \{0,1\}}T_{\sigma}\big(\mathbf{f}_1^{(\mu_1)},\dots,\mathbf{f}_l^{(\mu_l)} \big).
 \end{align*}
 
 We first estimate the term associated with $T_{\sigma}\big(\mathbf{f}_1^{(0)},\dots,\mathbf{f}_l^{(0)} \big)$.
 By applying the $L^r\times\cdots\times L^r\to L^{r/l}$ boundedness for $T_{\sigma}$ in Proposition \ref{multilineargerho} with 
 $$m=-\frac{nl}{r}(1-\rho)<-n(1-\rho)\Big(\frac{l}{r}-\frac{1}{2} \Big)=m_{\rho}(r,\dots,r),$$ we obtain
 \begin{align*}
 \Big(\frac{1}{|Q|}\int_Q \big|T_{\sigma}\big(\mathbf{f}_1^{(0)},\dots,\mathbf{f}_l^{(0)}\big)(y) \big|^{\frac{r}{l}} \; dy \Big)^{\frac{l}{r}}&\le \frac{1}{|Q|^{\frac{l}{r}}}\big\Vert T_{\sigma}\big(\mathbf{f}^{(0)}_1,\dots,\mathbf{f}_l^{(0)}\big)\big\Vert_{L^{\frac{r}{l}}(\bbrn)}\\
 &\lesssim \frac{1}{\ell(Q)^{\frac{nl}{r}}}\prod_{j=1}^{l}\Vert f_j\Vert_{L^r(Q^{**})}
 \end{align*}
 and this is obviously controlled by $\mathbf{M}_r\big(f_1,\dots,f_l\big)(x)$ by using the argument in \eqref{l2onpest} with $P=Q^{**}$.

 For the remaining $2^l-1$ terms, we employ the decomposition \eqref{sigmadecomp} and the inequality \eqref{keypropo2_1}.
 Indeed, assuming that
$(\mu_1,\dots,\mu_l)\in \{0,1\}^l\setminus \{(0,\dots,0)\}$,
 for $y\in Q$ and $N>\frac{nl}{r}$,
 \begin{align*}
 \Big| T_{\sigma}\big(\mathbf{f}_1^{(\mu_1)},\dots,\mathbf{f}_l^{(\mu_l)} \big)(y)\Big|&\le \sum_{k=0}^{\infty}\Big| T_{\sigma_k}\big(\mathbf{f}_1^{(\mu_1)},\dots,\mathbf{f}_l^{(\mu_l)} \big)(y)\Big|\\
 &\lesssim \sum_{k=0}^{\infty} \big(2^{k\rho}\ell(Q)\big)^{-(N-\frac{nl}{r})}\mathbf{M}_r\big(f_1,\dots,f_l\big)(x)\\
 &\sim \ell(Q)^{-(N-\frac{nl}{r})}\mathbf{M}_r\big(f_1,\dots,f_l\big)(x)\le \mathbf{M}_r\big(f_1,\dots,f_l\big)(x).
 \end{align*} 
 This completes the proof of \eqref{mainestimate0}.

 \subsection{Proof of \eqref{mainestimate}}
 Assume $\ell(Q)<1$. As in the linear case in Section \ref{linearproof}, we will treat  separately the two cases $0< \rho<\frac{r}{2l}$ and $\frac{r}{2l}\le \rho<1$.\\
 
 {\bf The case when $0< \rho<\frac{r}{2l}$}.\\
Let $Q^*_{\rho}$ be the concentric dilate of $Q$ with $\ell(Q_{\rho}^*)=10\sqrt{n} l \ell(Q)^{\rho}$.
We write
\begin{equation}\label{decomfqrho}
f_j=f_j\chi_{Q_{\rho}^*}+f_j\chi_{(Q_{\rho}^*)^c}=:\mathtt{f}_j^{(0)}+\mathtt{f}_j^{(1)} \q \text{ for }~j=1,\dots,l.
\end{equation}
To establish \eqref{mainestimate}, we need to estimate
\begin{equation*}
 \mathcal{I}_1:=\Big(\frac{1}{|Q|}\int_Q \big| T_{\sigma}\big( \mathtt{f}_1^{(0)},\dots, \mathtt{f}_l^{(0)}\big)(y)\big|^{\frac{r}{l}}\; dy\Big)^{\frac{l}{r}}
\end{equation*}
and
\begin{equation*}
\mathcal{I}_2:=\inf_{c_Q\in\mathbb{C}}\Big(\frac{1}{|Q|}\int_Q  \Big|\Big(T_{\sigma}\big(f_1,\dots,f_l\big)(y)-T_{\sigma}\big( \mathtt{f}_1^{(0)},\dots,\mathtt{f}_l^{(0)}\big)(y)\Big)-c_Q\Big|^{\frac{r}{l}} \; dy\Big)^{\frac{l}{r}}.
\end{equation*}
Using successively H\"older's inequality, Lemma \ref{Naibopropo}, and the inequality \eqref{l2onpest}, we obtain
 \begin{align*}
 \mathcal{I}_1&\le \frac{1}{|Q|^{\frac{l\rho}{r}}}\big\Vert    T_{\sigma}\big( \mathtt{f}_1^{(0)},\dots,\mathtt{f}_l^{(0)}\big)  \big\Vert_{L^{\frac{r}{l\rho}}(\bbrn)}\lesssim \frac{1}{\ell(Q)^{\frac{nl\rho }{r}}}\prod_{j=1}^{l}\Vert f_j\Vert_{L^r(Q_{\rho}^*)}\lesssim \mathbf{M}_r\big(f_1,\dots,f_l \big)(x),
 \end{align*}
 as desired.
 For the term $\mathcal{I}_2$, we utilize the decomposition \eqref{sigmadecomp} and take
  \begin{equation*}
 c_Q=\sum_{k:2^k\ell(Q)<1}\Big( T_{\sigma_k}\big(f_1,\dots,f_l\big)(x)-T_{\sigma_k}\big(\mathtt{f}_1^{(0)},\dots,\mathtt{f}_l^{(0)}\big)(x)\Big)
 \end{equation*}
to deduce
  \begin{align*}
& \Big|\Big(T_{\sigma}\big(f_1,\dots,f_l\big)(y)-T_{\sigma}\big( \mathtt{f}_1^{(0)},\dots,f_l^{(0)}\big)(y)\Big)-c_Q\Big|\\
 &\le \sum_{k:2^k\ell(Q)\ge 1}\Big|  T_{\sigma_k}\big(f_1,\dots,f_l\big)(y)-T_{\sigma_k}\big(\mathtt{f}_1^{(0)},\dots,\mathtt{f}_l^{(0)}\big)(y)   \Big|\\
 &\q +\sum_{ \substack{\mu_1,\dots,\mu_l\in \{0,1\}\\ (\mu_1,\dots,\mu_l)\not= (0,\dots,0) }    }\sum_{k:2^k\ell(Q)<1}\Big| T_{\sigma_k}\big( \mathtt{f}_1^{(\mu_1)},\dots,\mathtt{f}_l^{(\mu_l)}\big)(y) -T_{\sigma_k}\big( \mathtt{f}_1^{(\mu_1)},\dots,\mathtt{f}_l^{(\mu_l)}\big)(x)    \Big|\\
 &=:\mathscr{I}_1+\mathscr{I}_2.
 \end{align*}
Here, we used the fact that
\begin{equation}\label{tsikfgdecomp}
 T_{\sigma_k}\big(f_1,\dots,f_l\big)-T_{\sigma_k}\big( \mathtt{f}_1^{(0)},\dots,\mathtt{f}_l^{(0)}\big)=\sum_{ \substack{\mu_1,\dots,\mu_l\in \{0,1\}\\ (\mu_1,\dots,\mu_l)\not= (0,\dots,0) }    }T_{\sigma_k}\big( \mathtt{f}_1^{(\mu_1)},\dots,\mathtt{f}_l^{(\mu_l)}\big).
 \end{equation}
 We first claim that
 \begin{equation}\label{sk2klqge1}
 \mathscr{I}_1 \lesssim \mathbf{M}_r\big(f_1,\dots,f_l \big)(x).
 \end{equation}
 To verify this, using \eqref{tsikfgdecomp} again, 
 let us bound the summand in $\mathscr{I}_1$ by
 $$\sum_{ \substack{\mu_1,\dots,\mu_l\in \{0,1\}\\ (\mu_1,\dots,\mu_l)\not= (0,\dots,0) }    }  \Big| T_{\sigma_k}\big( \mathtt{f}_1^{(\mu_1)},\dots,\mathtt{f}_l^{(\mu_l)}\big)(y)\Big|  $$
 and this can be actually controlled by
 $$\big( 2^k\ell(Q)\big)^{-\rho(N-\frac{nl}{r})}\mathbf{M}_r\big(f_1,\dots,f_l \big)(x)$$
 for $N>\frac{nl}{r}$,  in view of \eqref{keypropo2_1}.
 Since $\rho>0$, by taking the sum over $k$ with $2^k\ell(Q)\ge 1$, we can prove \eqref{sk2klqge1}.
For the term $ \mathscr{I}_2$, we choose 
 $$\frac{nl}{r}<N<\frac{nl}{r}+\frac{1}{\rho}$$
 and then apply \eqref{keypropo2_2} to obtain
 \begin{equation}\label{sk2klqge2}
 \mathscr{I}_2\lesssim \sum_{k:2^k\ell(Q)<1}  \big( 2^k\ell(Q)\big)^{1-\rho(N-\frac{nl}{r})} \mathbf{M}_r\big(f_1,\dots,f_l\big)(x) \sim \mathbf{M}_r\big(f_1,\dots,f_l\big)(x)  
 \end{equation}
 as $1-\rho(N-\frac{nl}{r})>0$.
Now it follows from \eqref{sk2klqge1} and \eqref{sk2klqge2} that
 $$\mathcal{I}_2\lesssim \mathbf{M}_r\big(f_1,\dots,f_l\big)(x),$$
 which completes the proof of \eqref{mainestimate} for $0< \rho<\frac{r}{2l}$.\\

   {\bf The case when $\frac{r}{2l}\le \rho<1$}.\\
Similar to the linear case in Section \ref{linearproof},  the proof of \eqref{mainestimate} for $\frac{r}{2l}\le \rho<1$ will be divided into three parts based on the range of $k$.
It will suffice to estiamte
 \begin{align*}
 \mathcal{J}_1&:=\inf_{c_Q\in \bbc}\Big(\frac{1}{|Q|}\int_Q \Big| \sum_{k:2^k\ell(Q)< 1}T_{\sigma_k}\big(f_1,\dots,f_l\big)(y)-c_{Q}\Big|^{\frac{r}{l}} \; dy\Big)^{\frac{l}{r}},\\
\mathcal{J}_2&:=\Big(\sum_{\substack{k:2^k\ell(Q)\ge 1,\\~ 2^{\rho k       }\ell(Q)< 1}}\frac{1}{|Q|}\int_Q \big| T_{\sigma_k}\big( f_1,\dots,f_l\big)(y)\big|^{\frac{r}{l}}\;dy\Big)^{\frac{l}{r}},\\
 \mathcal{J}_3&:= \Big(\sum_{k:2^{\rho k}\ell(Q)\ge 1}\frac{1}{|Q|}\int_Q \big| T_{\sigma_k}\big( f_1,\dots,f_l\big)(y)\big|^{\frac{r}{l}}\;dy\Big)^{\frac{l}{r}}
 \end{align*}
where we note that $\frac{r}{l} \le 1$.
Then \eqref{mainestimate} follows from the estimates
\begin{equation}\label{iiestm2}
\mathcal{J}_{\mathrm{i}}\lesssim \mathbf{M}_r\big(f_1,\dots,f_l\big)(x),\q \mathrm{i}=1,2,3.
\end{equation}
To achieve this, we choose a number $\lambda$ such that
 $$\frac{2\rho l-r}{2l-r}<\lambda<\rho,$$ which allows us to employ Lemma \ref{rhoge12}.

We begin with $\mathcal{J}_1$. 
For this one, we use the decomposition \eqref{decomfqrho} again, recalling $Q_{\rho}^*$ is the concentric dilate of $Q$ whose side-length is $10\sqrt{n} l \ell(Q)^{\rho}$. 
Then we write
 \begin{align*}
T_{\sigma_k}(f_1,\dots,f_l)&=T_{\sigma_k}\big(\mathtt{f}_1^{(0)},\dots,\mathtt{f}_l^{(0)}\big)+\sum_{ \substack{\mu_1,\dots,\mu_l\in \{0,1\}\\ (\mu_1,\dots,\mu_l)\not= (0,\dots,0) }    }T_{\sigma_k}\big(\mathtt{f}_1^{(\mu_1)},\dots,\mathtt{f}_l^{(\mu_l)} \big)\\
&=:\mathscr{F}_k^{(1)}+\mathscr{F}_k^{(2)}.
 \end{align*}
Setting
 $$c_Q=\sum_{k:2^k\ell(Q)<1}\mathscr{F}_k^{(2)}(x),$$
 the term $ \mathcal{J}_1$ can be controlled by
 \begin{equation*}
  \bigg[\sum_{k:2^k\ell(Q)<1}\Big(\frac{1}{|Q|}\int_Q \big| \mathscr{F}_k^{(1)}(y)\big|^{\frac{r}{l}}\; dy+\frac{1}{|Q|}\int_Q \big|\mathscr{F}_k^{(2)}(y)-\mathscr{F}_k^{(2)}(x) \big|^{\frac{r}{l}}\;dy\Big)\bigg]^{\frac{l}{r}}.
 \end{equation*}
 First of all, Proposition \ref{keypropo1} (2) yields that
\begin{equation*}
\Big(\frac{1}{|Q|}\int_Q \big| \mathscr{F}_k^{(1)}(y)\big|^{\frac{r}{l}}\; dy\Big)^{\frac{l}{r}}\le \frac{1}{|Q|}\int_Q \big| \mathscr{F}_k^{(1)}(y)\big|\; dy\lesssim \big(2^k\ell(Q)\big)^{\lambda n \frac{l(1-\rho)}{r(1-\lambda)} }\mathbf{M}_r\big(f_1,\dots,f_l\big)(x).
\end{equation*}
Moreover, using \eqref{keypropo2_2}, for $y\in Q$
 \begin{equation*}
 \big|\mathscr{F}_k^{(2)}(y)-\mathscr{F}_k^{(2)}(x) \big|\lesssim \big(2^k\ell(Q) \big)^{1-\rho(N-\frac{nl}{r})}\mathbf{M}_r\big(f_1,\dots,f_l \big)(x)
 \end{equation*}
 where we choose $\frac{nl}{r}<N<\frac{nl}{r}+\frac{1}{\rho}$. This proves \eqref{iiestm2} for $\mathrm{i}=1$.\\

 For the term $\mathcal{J}_2$, assume that
 \begin{equation*}%\label{musel}
 1\le 2^k\ell(Q) \q \text{ and }\q 2^{\rho k}\ell(Q)<1.
 \end{equation*}
  We choose a positive number $\epsilon$ such that
 \begin{equation*}%\label{rangeep}
 \lambda \Big(\frac{1-\rho}{1-\lambda}\Big)<\epsilon<\rho,
 \end{equation*}
 which is possible because $\lambda<\rho$.
 Now let $Q_{\epsilon,k}^{*,\rho}$ be the concentric dilate of $Q$ with $\ell(Q_{\epsilon,k}^{*,\rho})=10\sqrt{n} l\ell(Q)^{\rho} \big(2^{k}\ell(Q) \big)^{-\epsilon}$. 
 Similar to \eqref{10sqnlqlelqrho}, it can be easily verified that
 $$10 \sqrt{n} l \ell(Q)\le \ell(Q_{\epsilon,k}^{*,\rho}).$$
Then we divide
 $$f_j=f_j\chi_{Q_{\epsilon,k}^{*,\rho}}+f_j\chi_{(Q_{\epsilon,k}^{*,\rho})^c}=:f_{j,\rho,k}^{(0)}+f_{j,\rho,k}^{(1)} \q \text{ for }~j=1,\dots,l$$
  so that $T_{\sigma_k}(f_1,\dots,f_l)$ can be written as
  \begin{align*}
  T_{\sigma_k}(f_1,\dots,f_l)&=T_{\sigma_k}\big(f_{1,\rho,k}^{(0)},\dots,f_{l,\rho,k}^{(0)}\big)+   \sum_{ \substack{\mu_1,\dots,\mu_l\in \{0,1\}\\ (\mu_1,\dots,\mu_l)\not= (0,\dots,0) }    }T_{\sigma_k}\big(f_{1,\rho,k}^{(\mu_1)},\dots,f_{l,\rho,k}^{(\mu_l)}\big)\\
  &=: \mathfrak{F}_k^{(1)}+\mathfrak{F}_k^{(2)}.
  \end{align*}
 Proposition \ref{keypropo1} (2) first yields that
\begin{align*}
\frac{1}{|Q|}\int_Q \big| \mathfrak{F}_k^{(1)}(y)\big|\; dy &\lesssim \big( 2^{ k}\ell(Q)\big)^{ \lambda n \frac{l(1-\rho)}{r(1-\lambda)}    }\big(2^k\ell(Q) \big)^{-\frac{\epsilon n l}{r}}\mathbf{M}_r\big(f_1,\dots,f_l \big)(x)\\
&=\big(2^k\ell(Q) \big)^{-\frac{nl}{r}(\epsilon-\lambda(\frac{1-\rho}{1-\lambda}))}\mathbf{M}_r\big(f_1,\dots,f_l \big)(x).
\end{align*}
Furthermore, using \eqref{keypropo2_1}, we have
\begin{align*}
\big| \mathfrak{F}_k^{(2)}(y)\big|& \lesssim \Big( 2^{k\rho}\ell(Q)^{\rho}\big( 2^k\ell(Q)\big)^{-\epsilon}    \Big)^{-(N-\frac{nl}{r})}\mathbf{M}_r\big(f_1,\dots,f_l \big)(x)\\
&= \big(2^k\ell(Q) \big)^{-(\rho-\epsilon)(N-\frac{nl}{r})} \mathbf{M}_r\big(f_1,\dots,f_l \big)(x)
\end{align*} where $N>\frac{nl}{r}$.
These conclude \eqref{iiestm2} for $\ii=2$.\\

Next let us consider the last one $\mathcal{J}_3$. In this case, we will use the decomposition \eqref{dildecfg} to write
 \begin{align*}
 T_{\sigma_k}(f_1,\dots,f_l)&=T_{\sigma_k}\big(\mathbf{f}_1^{(0)},\dots,\mathbf{f}_l^{(0)} \big)+  \sum_{ \substack{\mu_1,\dots,\mu_l\in \{0,1\}\\ (\mu_1,\dots,\mu_l)\not= (0,\dots,0) }    }      T_{\sigma_k}\big(\mathbf{f}_1^{(\mu_1)},\dots,\mathbf{f}_l^{(\mu_l)} \big)            \\
 &=:\mathbf{F}_k^{(1)}+\mathbf{F}_k^{(2)}.
 \end{align*}
We notice that
 $$2^{\frac{kn}{4}(1-\rho)}\sigma_k=2^{-k[m-m_{\rho}(r,\dots,r)+\frac{n}{4}(1-\rho)]}\sigma_k\in \mathbb{M}_lS_{\rho,\rho}^{m_{\rho}(r,\dots,r)-\frac{n}{4}(1-\rho)}(\bbrn)$$
and thus it follows from Proposition \ref{multilineargerho} (with $p_1=\dots=p_l=r$) that
\begin{align*}
 \Big( \frac{1}{|Q|}\int_Q \big| \mathbf{F}_k^{(1)}(y)\big|^{\frac{r}{l}}\; dy\Big)^{\frac{l}{r}}&\le \frac{1}{|Q|^{\frac{l}{r}}} \big\Vert T_{\sigma_k}(\mathbf{f}_1^{(0)},\dots,\mathbf{f}_l^{(0)})\big\Vert_{L^{\frac{r}{l}}(\bbrn)}\\
 &\lesssim \ell(Q)^{-\frac{nl}{r}}2^{-\frac{kn}{4}(1-\rho)}\prod_{j=1}^{l}\Vert f_j\Vert_{L^r(Q^{**})}.
 \end{align*}
 Clearly, 
 \begin{equation*}
\prod_{j=1}^{l}\Vert f_j\Vert_{L^r(Q^{**})} \lesssim \ell(Q)^{\frac{nl}{r}}\mathbf{M}_r\big( f_1,\dots,f_l\big)(x),
 \end{equation*}
 and thus have
 \begin{equation*}%\label{fk1est1}
 \Big( \frac{1}{|Q|}\int_Q \big| \mathbf{F}_k^{(1)}(y)\big|^{\frac{r}{l}}\; dy\Big)^{\frac{l}{r}}\lesssim 2^{-\frac{kn}{4}(1-\rho)}\mathbf{M}_r\big(f_1,\dots,f_l\big)(x).
 \end{equation*}
 This proves
 \begin{align*}
\Big(\sum_{k:2^{\rho k}\ell(Q)\ge 1}  \frac{1}{|Q|}\int_Q \big| \mathbf{F}_k^{(1)}(y)\big|^{\frac{r}{l}}\; dy\Big)^{\frac{l}{r}} &\le \Big(\sum_{k=0}^{\infty} \frac{1}{|Q|}\int_Q \big| \mathbf{F}_k^{(1)}(y)\big|^{\frac{r}{l}}\; dy\Big)^{\frac{l}{r}}\\
 &\lesssim \mathbf{M}_r\big(f_1,\dots,f_l\big)(x).
 \end{align*}
Moreover,  \eqref{keypropo2_1} implies
\begin{equation*}%\label{fk1est234}
\big| \mathbf{F}_k^{(2)}(y)\big| \lesssim \big( 2^{k\rho} \ell(Q)        \big)^{-(N-\frac{nl}{r})}\mathbf{M}_r\big(f_1,\dots,f_l \big)(x)
\end{equation*}
where $N>\frac{nl}{r}$. 
This  implies
\begin{align*}
&\Big(\sum_{k:2^{\rho k}\ell(Q)\ge 1}  \frac{1}{|Q|}\int_Q \big| \mathbf{F}_k^{(2)}(y)\big|^{\frac{r}{l}}\; dy\Big)^{\frac{l}{r}} \\
&\lesssim  \mathbf{M}_r\big(f_1,\dots,f_l \big)(x) \Big(\sum_{k:2^{\rho k}\ell(Q)\ge 1}\big( 2^{k\rho} \ell(Q)        \big)^{-(N-\frac{nl}{r})\frac{r}{l}}\Big)^{\frac{l}{r}}\\
&\lesssim \mathbf{M}_r\big(f_1,\dots,f_l \big)(x).
\end{align*}
Combining the estimates for both $\mathbf{F}_k^{(1)}$ and $\mathbf{F}_k^{(2)}$, we finally obtain
 $$\mathcal{J}_3\lesssim   \mathbf{M}_r\big(f_1,\dots,f_l \big)(x).$$
This completes the proof of \eqref{mainestimate}.

 \hfill

 \section{Proof of Theorem \ref{mainthm} (from Theorem \ref{mainpointest})}

We first observe that for any $0<t\le 1$,
 \begin{equation}\label{sharpmaxembed}
 \Big(\mathcal{M}^{\sharp}\big(|f|^t\big)(x)\Big)^{\frac{1}{t}}\lesssim \mathcal{M}_t^{\sharp}f(x)
 \end{equation}
as
 \begin{equation*}
 \Big(\frac{1}{|Q|}\int_Q \Big| \big|f(y)\big|^t-|c_Q|^t\Big|\; dy\Big)\le \Big(\frac{1}{|Q|}\int_Q \big| f(y)-c_Q\big|^t\; dy\Big).
 \end{equation*}

Now we point out that if $f_1,\dots,f_l$ are Schwartz functions on $\bbrn$, then
$T_{\sigma}(f_1,\dots,f_l)$ is also a Schwartz function, which can be verified easily by using integration by parts several times.
 Moreover, according to \cite[Theorem 3.6]{Le_Om_Pe_To_Tr2009},
  $$  v_{\vec{w}}\in A_{\frac{lp}{r}}\q     \text{for $(w_1,\dots,w_l)\in A_{\frac{p_1}{r},\dots,\frac{p_l}{r}}$}. $$
 Therefore, for $(w_1,\dots,w_l)\in A_{\frac{p_1}{r},\dots,\frac{p_l}{r}}$, we have
 $$\big| T_{\sigma}(f_1,\dots,f_l)\big|^{\frac{r}{l}}\in L^{\frac{lp}{r}}(v_{\vec{w}}),$$
 which implies that
 $$\mathrm{M}^{dyad}\Big( \big| T_{\sigma}(f_1,\dots,f_l)\big|^{\frac{r}{l}}\Big)\in L^{\frac{lp}{r}}(v_{\vec{w}})$$
 as $\frac{lp}{r}>1$,
 where $\mathrm{M}^{dyad}$ is the dyadic Hardy-Littlewood maximal operator, which is defined by  replacing the supremum in \eqref{hardymax} by the supremum over all dyadic cubes containing $x$, together with setting $r=1$.
 Then \cite[Lemma 7.10]{Duan} yields that
\begin{align*}
 \big\Vert T_{\sigma}(f_1,\dots,f_l)\big\Vert_{L^p(v_{\vec{w}})}&=\Big\Vert \big| T_{\sigma}(f_1,\dots,f_l)\big|^{\frac{r}{l}}\Big\Vert_{L^{\frac{lp}{r}}(v_{\vec{w}})}^{\frac{l}{r}} \\
 &\le \Big\Vert \mathrm{M}^{dyad}\Big( \big| T_{\sigma}(f_1,\dots,f_l)\big|^{\frac{r}{l}}\Big) \Big\Vert_{L^{\frac{lp}{r}}(v_{\vec{w}})}^{\frac{l}{r}}\\
 &\lesssim \Big\Vert \mathcal{M}^{\sharp}\Big( \big| T_{\sigma}(f_1,\dots,f_l)\big|^{\frac{r}{l}}\Big) \Big\Vert_{L^{\frac{lp}{r}}(v_{\vec{w}})}^{\frac{l}{r}}
 \end{align*}
where we used the fact that homogeneous sharp maximal function is less than inhomogeneous sharp maximal one.
In light of \eqref{sharpmaxembed}, the last displayed expression is bounded by
$$\Big\Vert \mathcal{M}^{\sharp}_{r/l} \big( T_{\sigma}(f_1,\dots,f_l)\big)\Big\Vert_{L^p(v_{\vec{w}})}$$
 and then this can be dominated by
$$ \big\Vert   \mathbf{M}_r(f_1,\dots,f_l)\big\Vert_{L^p(v_{\vec{w}})}$$
thanks to Theorem  \ref{mainpointest}. Finally, Theorem \ref{keylemma} proves the above $L^p$ norm is bounded by
 $$\prod_{j=1}^{l}\Vert f_j\Vert_{L^{p_j}(w_j)},$$
 which finishes the proof of Theorem \ref{mainthm}.

 \hfill

 \appendix

 \section{Proof of \eqref{linearckins}}
 
 Let $\alpha,\beta\in (\bbn_0)^n$ be any multi-indices and $k\in  \bbn_0$.
 We first note that
\begin{align*}
&\big| \partial_{x}^{\alpha}\partial_{\xi}^{\beta}\sigma_k(x,\xi)\big|\lesssim  \big( 1+|\xi|\big)^{m-\rho(|\beta|-|\alpha|)}\chi_{ \{1+|\xi|\sim 2^k\}}\sim 2^{k(m-\rho(|\beta|-|\alpha|))}\chi_{ \{1+|\xi|\sim 2^k\}}
\end{align*}
where the constants in the estimates are independent of $k$.
This yields that
 \begin{align*}
 \big| \partial_{x}^{\alpha}\partial_{\xi}^{\beta}c_k(x,\xi)\big| &=2^{\lambda k (|\beta|-|\alpha|)}\big| \partial_{x}^{\alpha}\partial_{\xi}^{\beta}\sigma_k( 2^{-\lambda k}x,2^{\lambda k}\xi)\big|\, \chi_{ \{1+|\xi|\sim 2^{(1-\lambda)k\}}}\\
 &\lesssim 2^{\lambda k (|\beta|-|\alpha|)}  2^{k(m-\rho(|\beta|-|\alpha|))}    \chi_{ \{1+|\xi|\sim 2^{(1-\lambda)k}\}}\\
 &= 2^{k[m-(\rho-\lambda)(|\beta|-|\alpha|)]} \chi_{ \{1+|\xi|\sim 2^{(1-\lambda)k}\}}\\
 &\lesssim \big( 1+|\xi|\big)^{\frac{m}{1-\lambda}-(\frac{\rho-\lambda}{1-\lambda})(|\beta|-|\alpha|)}
 \end{align*}
 uniformly in $k$,
 as
 $$2^k\sim (1+|\xi|)^{\frac{1}{1-\lambda}}.$$
 This proves \eqref{linearckins}.

%\bibliographystyle{apalike}
%\bibliography{HeBib}

%\printindex

\end{document}